\documentclass[11pt, a4paper]{amsart}

\usepackage{amsmath}
\usepackage{amsfonts}
\usepackage{amssymb}
\usepackage{amsthm}
\usepackage{mathrsfs}
\usepackage{graphics}
\usepackage[all]{xy}
\usepackage{mathtools}
\newcommand{\seq}{\coloneqq}
\usepackage[
bookmarks=false,
colorlinks=true,
debug=true,
naturalnames=true,
pdfnewwindow=true,
citecolor=blue,
linkcolor=blue,
urlcolor = blue]{hyperref}
\usepackage{comment}
\mathtoolsset{showonlyrefs=true}
\usepackage{yfonts}

\newenvironment{red}{\relax\color{red}}{\relax}
\newenvironment{blue}{\relax\color{blue}}{\relax}

\newcommand{\ber}{\begin{red}}
\newcommand{\er}{\end{red}}
\newcommand{\beb}{\begin{blue}}
\newcommand{\eb}{\end{blue}}

\newtheorem{Thm}{Theorem}[section]
\newtheorem{Cor}[Thm]{Corollary}
\newtheorem{Prop}[Thm]{Proposition}
\newtheorem{Lem}[Thm]{Lemma}

\theoremstyle{definition}
\newtheorem{Def}[Thm]{Definition}
\newtheorem{Rem}[Thm]{Remark}
\newtheorem{Ex}[Thm]{Example}
\newtheorem{Ques}[Thm]{Question}

\numberwithin{equation}{section}

\newcommand{\arxiv}[1]{\href{http://arxiv.org/abs/#1}{\texttt{arXiv:#1}}}

\newcommand{\ol}{\overline}
\newcommand{\ul}{\underline}

\newcommand{\ep}{\epsilon}
\newcommand{\vep}{\varepsilon}

\newcommand{\Hom}{\mathop{\mathrm{Hom}}\nolimits}

\newcommand{\Aut}{\mathop{\mathrm{Aut}}\nolimits}
\newcommand{\End}{\mathop{\mathrm{End}}\nolimits}
\newcommand{\Ext}{\mathop{\mathrm{Ext}}\nolimits}
\newcommand{\Ker}{\mathop{\mathrm{Ker}}\nolimits}
\newcommand{\Cok}{\mathop{\mathrm{Coker}}\nolimits}
\newcommand{\Image}{\mathop{\mathrm{Im}}\nolimits}
\newcommand{\Res}{\mathop{\mathrm{Res}}\nolimits}
\newcommand{\Infl}{\mathop{\mathrm{Infl}}\nolimits}
\newcommand{\For}{\mathop{\mathrm{For}}\nolimits}

\newcommand{\Spec}{\mathop{\mathrm{Spec}}\nolimits}

\newcommand{\vdim}{\mathop{\ul{\dim}}}

\newcommand{\fg}{\mathfrak{g}}

\newcommand{\dd}{\mathfrak{d}}
\newcommand{\ee}{\mathfrak{e}}
\newcommand{\fA}{\mathfrak{A}}
\newcommand{\oo}{\mathfrak{o}}

\newcommand{\N}{\mathbb{N}}
\newcommand{\Z}{\mathbb{Z}}
\newcommand{\Q}{\mathbb{Q}}
\newcommand{\C}{\mathbb{C}}

\newcommand{\kk}{\Bbbk}

\newcommand{\Cc}{\mathscr{C}}

\newcommand{\cL}{\mathcal{L}}

\newcommand{\cM}{\mathcal{M}}

\newcommand{\rH}{\mathrm{H}}

\newcommand{\tB}{\tilde{B}}

\newcommand{\wh}{\widehat}

\newcommand{\hrH}{\wh{\rH}}

\newcommand{\Perv}{\mathop{\mathrm{Perv}}}

\newcommand{\id}{\mathsf{id}}

\newcommand{\qv}{\mathfrak{M}^\bullet}

\newcommand{\pt}{\mathrm{pt}}

\newcommand{\IC}{\mathrm{IC}}

\newcommand{\Gr}{\mathop{\mathrm{Gr}}\nolimits\!}

\newcommand{\II}{\mathtt{I}}

\newcommand{\rep}{\mathop{\mathsf{rep}}\nolimits}
\newcommand{\inj}{\mathop{\mathsf{inj}}\nolimits}
\newcommand{\NII}{\N^{\II \sqcup \II}}

\newcommand{\gV}{\mathscr{V}^\bullet}
\newcommand{\gdim}{\mathop{\mathrm{gdim}}\nolimits}
\newcommand{\bv}{\mathbf{v}}
\newcommand{\Dom}{\Lambda^+}
\newcommand{\tw}{\tilde{w}}
\newcommand{\uk}{\ul{\kk}}
\newcommand{\Lfg}{L\fg}
\newcommand{\tchi}{\tilde{\chi}}

\makeatletter
\@namedef{subjclassname@2020}{%
  \textup{2020} Mathematics Subject Classification}
\makeatother

\title[R-matrices and E-invariants]{Singularities of normalized R-matrices and E-invariants for Dynkin quivers}
\date{\today}

\author{Ryo Fujita}
\address{Research Institute for Mathematical Sciences, Kyoto University, Kitashirakawa-Oiwake-cho, Sakyo, Kyoto, 606-8502, Japan}
\email{rfujita@kurims.kyoto-u.ac.jp}

\begin{document}

\maketitle

\begin{abstract}
We study the singularities of normalized $R$-matrices between arbitrary simple modules over the quantum loop algebra of type ADE in Hernandez--Leclerc's level-one subcategory using equivariant perverse sheaves, following the previous works by Nakajima [Kyoto J.\ Math.\ 51(1), 2011] and Kimura--Qin [Adv.\ Math.\ 262, 2014]. 
We show that the pole orders of these $R$-matrices coincide with the dimensions of $E$-invariants between the corresponding decorated representations of Dynkin quivers.  
This result can be seen as a correspondence of numerical characteristics between additive and monoidal categorifications of cluster algebras of finite ADE type. 
\end{abstract}

\tableofcontents

\section{Introduction}
\subsection{}
The quantum loop algebra $U_q(\Lfg)$ associated with a complex simple Lie algebra $\fg$ was introduced in mid 80s as the symmetry of certain quantum integrable systems and solvable lattice models in theoretical physics.
It is a Hopf algebra deformation of the universal enveloping algebra of the loop Lie algebra $\Lfg = \fg[z^{\pm 1}]$.
The category $\Cc$ of finite-dimensional representations of $U_q(L\fg)$ exhibits a very interesting monoidal structure and has been studied intensively for several decades. 

One of the remarkable features of the monoidal category $\Cc$ is that it is not braided, in contrast to that of finite-dimensional representations of $U_q(\fg)$, but it is ``generically braided'' in the following sense.
Throughout the paper, we assume that the quantum parameter $q$ is generic.
For two simple objects $L$ and $L'$ of $\Cc$, the tensor product $L \otimes L'$ sometimes fails to be isomorphic to the opposite product $L' \otimes L$. 
However, if we replace $L'$ with its deformation $L'(z)$ with a generic spectral parameter $z$,  there always exists a unique isomorphism 
\[ R_{L,L'}(z) \colon L\otimes L'(z) \xrightarrow{\simeq} L'(z) \otimes L\]
called the \textit{normalized $R$-matrix} between $L$ and $L'$.
It can be seen as a matrix-valued rational function in $z$, and hence one can talk about its singularities. 
Let $\oo(L,L')$ denote the pole order of $R_{L,L'}(z)$ at $z=1$.
If both $R_{L,L'}(z)$ and $R_{L',L}(z)$ are regular at $z=1$, i.e., if $\oo(L,L') = \oo(L',L) = 0$ holds, the objects $L$ and $L'$ commute in $\Cc$ and the specialization $R_{L,L'}(z)|_{z=1}$ gives an isomorphism $L \otimes L' \simeq L' \otimes L$.
Thus, one can think of the pole order $\oo(L,L')$ as a measure of the non-commutativity between $L$ and $L'$.
In fact, it plays a key role in the recent studies on the category $\Cc$, especially in the theory of monoidal categorification of cluster algebras \cite{KKOP} and in the construction of generalized quantum affine Schur--Weyl duality functors \cite{KKK}. 

Despite its importance, computing the pole order $\oo(L,L')$ for general simple objects $L$ and $L'$ seems to be a difficult problem. 
Explicit computations have been accomplished for fundamental modules and partially for Kirillov--Reshetikhin modules.
See \cite[Appendix~A]{KKOPd} and \cite{OS} for a list of known computations.
Beyond these special classes, no systematic computations are known at this moment as far as the author understands. 
The purpose of this paper is to provide a description of the pole orders for another class of simple modules in relation with the cluster structure of $\Cc$.

\subsection{}
To state our main result in a precise manner, we need additional terminologies. 
From now on, we assume that $\fg$ is of type ADE, and let $Q$ be a Dynkin quiver of the same type as $\fg$.
In their seminal works \cite{HL, HL2}, Hernandez--Leclerc introduced a certain monoidal subcategory $\Cc_1$ of $\Cc$, depending on (a height function of) the quiver $Q$, which we call the \textit{level-one subcategory}.
They conjectured, and verified in several cases, that it gives a \textit{monoidal categorification} of a cluster algebra $A $ of finite ADE type (the same type as $\fg$), in the sense that there exists a ring isomorphism 
\[\tchi_q \colon K(\Cc_1) \xrightarrow{\simeq} A\]
from {the} Grothendieck ring $K(\Cc_1)$ to the cluster algebra $A$ through which the basis of $K(\Cc_1)$ formed by the simple isomorphism classes corresponds to the basis of $A$ formed by the cluster monomials.
The isomorphism $\tchi_q$ is given explicitly by the truncated $q$-character map of \cite{HL}.
The conjecture was later verified by Nakajima \cite{NakCA} and Kimura--Qin \cite{KQ} in full generality using the perverse sheaves on Nakajima's graded quiver varieties.

On the other hand, there is another kind of categorification of $A$, sometimes called an \textit{additive categorification}.  
Here, we use the version due to Caldero--Chapoton \cite{CC} and Derksen--Weyman--Zelevinsky \cite{DWZ} in terms of \textit{decorated representations} of the Dynkin quiver $Q$.
Recall that a decorated representation of $Q$ is a pair $\mathcal{M} = (M,V)$ of a usual representation $M$ of $Q$ (over $\C$) and a finite-dimensional $\II$-graded $\C$-vector space $V$, where $\II$ is the set of vertices of $Q$.
For two such pairs $\cM = (M,V)$ and $\cM' =(M', V')$, the \textit{$E$-invariant} between them is defined to be
\begin{equation} \label{eq:Edecr} 
E(\cM, \cM') \seq \Ext^1_Q(M,M') \oplus \bigoplus_{i \in \II}\Hom_\C(M_i, V'_i). 
\end{equation}
A decorated representation $\cM$ is said to be \textit{rigid} if $E(\cM, \cM) = 0$. 
The theory of additive categorification of $A$ tells us that there is a map
\[CC \colon \{ \text{decorated representations of $Q$} \} \to A\]
called the cluster character map (a.k.a.\ Caldero--Chapoton map) which satisfies $CC(\cM \oplus \cM') = CC(\cM) \cdot CC(\cM')$ and induces a bijection between the isomorphism classes of rigid decorated representations of $Q$ and the cluster monomials without frozen factors of $A$.

Now, we are ready to state our main result. 
It describes the pole orders of the normalized $R$-matrices in $\Cc_1$ in terms of the $E$-invariants for the Dynkin quiver $Q$. 
\begin{Thm} \label{Thm:main1}
For any simple objects $L$ and $L'$ of $\Cc_1$, we have
\[ \oo(L,L') = \dim E(\cM, \cM'), \]
where $\cM$ and $\cM'$ are rigid decorated representations of $Q$ satisfying $\tchi_q(L) = CC(\cM)$ and $\tchi_q(L') = CC(\cM')$ up to frozen factors. 
\end{Thm}

To obtain the main result, we apply Nakajima's geometric construction of representations of $U_q(L\fg)$ \cite{Nak, NakT} and verify a slightly different but equivalent assertion (= Theorem~\ref{Thm:main}), where decorated representations are replaced with injective copresentations, following Derksen--Fei~\cite{DF}. 
Our proof is based on the key observation by Kimura--Qin \cite{KQ}, generalizing the one by Nakajima~\cite{NakCA}, that the graded quiver variety relevant to the category $\Cc_1$ is simply a vector space and its dual is identified with the space $X$ of injective copresentations of the Dynkin quiver $Q$.      
In the previous work \cite{KQ}, this fact was crucially used to relate the equivariant perverse sheaves on $X$ to the Grothendieck ring $K(\Cc_1)$ or rather its quantum deformation.
In this paper, we go one more step further to relate the geometry of $X$ directly to representations in $\Cc_1$.
Namely, we interpret the deformed tensor products of simple objects of $\Cc_1$ and the $R$-matrices between them (under a certain condition, see \S\ref{Ssec:proof}) in terms of canonical operations for the equivariant perverse sheaves on $X$.
T{hen, we find t}he $E$-invariant in question as a transversal slice in $X$.
{We remark that a similar geometric interpretation of $R$-matrices has been studied in the author's previous works \cite{Frmat, FH} for tensor products of fundamental modules.
In this paper, we study tensor products of  arbitrary simple modules in the category $\Cc_1$.}

\subsection{}
Let $x, x' \in A$ be two non-frozen cluster variables, $L, L'$ prime simple objects of $\Cc_1$, and $\cM, \cM'$ rigid indecomposable decorated representations of $Q$ satisfying $x = \tchi_q(L) = CC(\cM)$ and $x' = \tchi_q(L') = CC(\cM')$. 
The theory of additive/monoidal categorifications tells us that the following three conditions are mutually equivalent:
\begin{enumerate}
\item $x$ and $x'$ belong to a common cluster (i.e., $xx'$ is a cluster monomial);
\item $\dd(L,L') \seq \oo(L,L') + \oo(L',L) = 0$;
\item $\ee(\cM, \cM') \seq \dim E(\cM, \cM') + \dim E(\cM', \cM) = 0$.
\end{enumerate} 
The invariant $\dd(L,L')$ was originally introduced by Kashiwara--Kim--Oh--Park \cite{KKOP}. 
On the other hand, the invariant $\ee(\cM, \cM')$ was considered by Marsh--Reineke--Zelevinsky \cite{MRZ}, which is identical to Fomin--Zelevinsky's compatible degree in \cite{FZy}.  
Theorem~\ref{Thm:main1} above implies the correspondence of these two numerical characteristics: 
\begin{equation} \label{eq:d=e} 
\dd(L,L') = \ee(\cM, \cM'),
\end{equation}
which does not seem automatic from the known categorifications results.

{Thus, it would be interesting to ask the following:}
{
\begin{Ques}
Does Theorem~\ref{Thm:main1} or the equality~\eqref{eq:d=e} generalize beyond the category $\Cc_1$ to other monoidal categorifications of cluster algebras?
\end{Ques}
}
At least for Kirillov--Reshetikhin modules, known computations suggest that such a generalization is plausible \cite[\S5]{FM}. {For this case, the above question has already been formulated as \cite[Conjecture 5.17]{FM}}.   
Note that the left-hand side of \eqref{eq:d=e} makes {also} sense for graded modules over symmetric quiver Hecke algebras \cite{KKKOm} and for the coherent Satake category \cite{CW} as well.    

{We also remark that there recently appeared several studies on the connection between additive and modoidal categorifications of cluster algebras from slightly different viewpoints, including \cite{Contu} and \cite{BFL}.}

Finally, during the revision of the present paper, the author was informed that Cao establishes in the latest version of \cite{Cao} the correspondence \eqref{eq:d=e} in a general setting of monoidal categorification of cluster algebras. His proof is purely algebraic and based on a series of works by Kang, Kashiwara, Kim, Oh and Park, including \cite{KKKOm, KKOP}. 

\subsection{Organization}
The present paper is organized as follows. 
In \S\ref{Sec:alg}, we state the main theorem (= Theorem~\ref{Thm:main}) after reviewing some necessary backgrounds.
In \S\ref{Ssec:cluster}, we briefly explain its cluster theoretical interpretation to see that it is equivalent to the above Theorem~\ref{Thm:main1}.
In \S\ref{Sec:geom}, we summarize Nakajima's geometric construction of representations of $U_q(\Lfg)$.   
In the final \S\ref{Sec:proof}, we apply the materials from \S\ref{Sec:geom} to study representations in the category $\Cc_1$ and discuss the proof of the main theorem.

\section{Algebraic preliminaries and main theorem}
\label{Sec:alg}

In this section, some necessary algebraic preliminaries are recalled before we state our main theorem (= Theorem~\ref{Thm:main}) in \S\ref{Ssec:state}.
In \S\ref{Ssec:cluster}, we briefly explain a cluster theoretical interpretation of Theorem~\ref{Thm:main} to see that it is indeed equivalent to Theorem~\ref{Thm:main1} in Introduction. 

\subsection{Representations of quantum loop algebras} \label{Ssec:qloop}
Let $\fg$ be a complex simple Lie algebra and $U_q(\Lfg)$ the quantum loop algebra associated with $\fg$, which is a Hopf algebra defined over an algebraically closed field $\kk$ of characteristic $0$ with $q \in \kk^\times$ being a parameter.
In Drinfeld's presentation, it is generated by the elements $x_{i,n}^{\pm}, h_{i,m}, k_i^{\pm 1}$ with $i \in \II, n \in \Z, m \in \Z \setminus \{ 0 \}$, where $\II$ is a labeling set of simple roots of $\fg$. 
It is also defined as a subquotient of the untwisted quantum affine algebra $U_q(\hat{\fg})$, and carries a family of Chevalley generators $e_i, f_i, k_i^{\pm 1}$, where $i\in \II \sqcup\{0\}$.
 In this paper, we use the coproduct $\Delta$ of $U_q(L\fg)$ given by the formula
\[\Delta(e_i) = e_i \otimes k_i^{-1} + 1 \otimes e_i, \quad 
\Delta(f_i) = f_i \otimes 1 + k_i \otimes f_i, \quad 
\Delta(k_i^{\pm 1}) = k_i^{\pm 1} \otimes k_i^{\pm 1}\]
for each $i\in \II \sqcup\{0\}$.
This is the same convention as in \cite{KKKO, KKOP} for example.

Throughout this paper, the quantum parameter $q$ is assumed to be algebraically independent over $\Q$.
Let $\Cc$ be the category of finite-dimensional modules over $U_q(\Lfg)$ of type $\mathbf{1}$. 
This is a $\kk$-linear monoidal abelian category.
We often abbreviate $\otimes_\kk$ as $\otimes$.

For objects $M,N \in \Cc$, we say that $M$ and $N$ \textit{commute} if the tensor product $M \otimes N$ is isomorphic to the opposite product $N \otimes M$.
In the category $\Cc$, there are many pairs of objects which do not commute.
Nevertheless, the Grothendieck ring $K(\Cc)$ of $\Cc$ is known to be a commutative domain \cite{FR}. 
This particularly implies that two simple modules $M$ and $N$ commute if $M \otimes N$ is simple. 
In this case, we say that $M$ and $N$ \textit{strongly commute}.

The isomorphism classes of simple modules in the category $\Cc$ are parameterized by the multiplicative monoid $(1 + z\kk[z])^\II$ of $\II$-tuples of monic polynomials, called the {Drinfeld polynomials} \cite[Ch.~12]{CP}.
We denote by $L(\varpi)$ a simple module in $\Cc$ corresponding to $\varpi \in (1 + z\kk[z])^\II$.
In the terminology of \cite{CP}, it is an $\ell$-highest weight module of $\ell$-highest weight $\varpi$.
In particular, $L(\varpi)$ has a distinguished generating vector, called an $\ell$-highest weight vector, uniquely up to multiple in $\kk^\times$.
Note that the monoid $(1+z\kk[z])^\II$ is generated by the elements $\varpi_{i,a} \seq ((1-az)^{\delta_{i,j}})_{j \in \II}$ for $(i,a) \in I \times \kk^\times$, where $\delta_{i,j}$ is the Kronecker delta.

\subsection{Normalized R-matrices}\label{Ssec:Rmat}
Let $z$ be an indeterminate. 
For an object $M \in \Cc$, we can define a new action of $U_q(\Lfg)$ on the $\kk[z^{\pm 1}]$-module $M[z^{\pm 1}] \seq M \otimes \kk[z^{\pm 1}]$ by the formula: 
\[ x_{i, n}^{\pm}(v\otimes a) \seq x_{i,n}^{\pm}v \otimes z^{n}a, \quad 
k_i (v\otimes a) \seq k_iv \otimes a, \quad 
h_{i,m} (v\otimes a) \seq h_{i,m}v \otimes z^{m}a, \]
where $v \in M, a \in \kk[z^{\pm 1}]$.
The resulting $U_q(\Lfg)[z^{\pm 1}]$-module $M[z^{\pm 1}]$ is called the \textit{affinization} of $M$.
We set $M(z) \seq M[z^{\pm 1}] \otimes_{\kk[z^{\pm 1}]} \kk(z)$.
In what follows, we sometimes identify the subspace $M \otimes 1$ of $M[z^{\pm 1}]$ with $M$. 

For a pair $(M,N)$ of simple modules in $\Cc$,  with fixed $\ell$-highest weight vectors $v_M \in M$ and $v_N \in N$, the $U_q(\Lfg)(z)$-modules $M \otimes N(z) $ and $N(z) \otimes M$ are known to be simple {(cf.\ \cite[Theorem 4.4]{Chari})}, and therefore we have a unique isomorphism 
\[ R_{M,N}(z) \colon M \otimes N(z) \to N(z) \otimes M\]
satisfying $R_{M,N}(z)(v_M \otimes v_N) = v_N \otimes v_M$.
It is called the \textit{normalized $R$-matrix} between $M$ and $N$.
Viewing it as a matrix-valued rational function in $z$, one can talk about the order of its poles, which does not depend on the choice of $\ell$-highest weight vectors $v_M$ and $v_N$.
We define a non-negative integer $\oo(M,N)$ to be the pole order of $R_{M,N}(z)$ at $z=1$, and set
\[ \dd(M,N) \seq \oo(M,N) + \oo(N,M). \]
This is the same as the invariant introduced in \cite{KKOP} (cf.~\cite[Proposition~3.16]{KKOP}). 
One may understand that the number $\dd(M,N)$ measures the non-commutativity between $M$ and $N$ as the following proposition suggests. 
We say that a simple module in $\Cc$ is \textit{real} if it strongly commutes with itself. 
\begin{Prop}[{\cite[Corollary~3.17]{KKOP}}]
\label{Prop:KKOP}
Let $M$ and $N$ be simple modules in $\Cc$. 
Assume that at least one of them is real.
Then $M$ and $N$ strongly commute if and only if $\dd(M,N) = 0$.
\end{Prop}

The following property is used later in \S\ref{Ssec:proof}. 
\begin{Lem}[cf.\ {\cite[Corollary~3.11(ii)]{KKKO}}]
\label{Lem:omult}
Let $M_1, M_2, N$ be simple modules in $\Cc$.
Assume that $M_1$ and $M_2$ strongly commute. 
Then, we have 
\begin{align*}
\oo(M_1\otimes M_2, N) &= \oo(M_1, N) + \oo(M_2, N), \\
\oo(N,M_1\otimes M_2) &= \oo(N,M_1) + \oo(N, M_2).
\end{align*}
\end{Lem}

\begin{Rem} \label{Rem:Mu}
For our purpose, it is convenient to have the following characterization of the number $\oo(M,N)$, {which immediately follows from the definition through the Laurent expansion at $z=1$}.

Let us introduce another indeterminate $u$ and consider the ring $\kk[\![u]\!]$ of formal power series. 
Viewing $\kk[z^{\pm1}]$ as a subring of $\kk[\![u]\!]$ by $z = e^u$, we define the infinitesimal deformation of $M$ to be
\[ M[\![u ]\!] \seq M[z^{\pm 1}] \otimes_{\kk[z^{\pm 1}]} \kk[\![u]\!]. \]
This is a $U_q(\Lfg)[\![u]\!]$-module.
By localization, we also get a $U_q(\Lfg)(\!(u)\!)$-module $M(\!(u)\!)$.
Note that $M[\![u ]\!]$ is a $\kk[\![u]\!]$-lattice of $M(\!(u)\!)$.  

For simple modules $M$, $N$ in $\Cc$, the normalized $R$-matrix $R_{M,N}(e^u)$ induces an isomorphism 
\[ \wh{R}_{M,N} \colon M \otimes N(\!(u)\!) \to N(\!(u)\!) \otimes M \]
of $U_q(\Lfg)(\!(u)\!)$-modules satisfying $\wh{R}_{M,N}(v_M \otimes v_N) = v_N \otimes v_M$.  
Then, the number $\oo(M,N)$ is equal to the non-negative integer $d$ such that we have
\[ u^d \wh{R}_{M,N}(M \otimes N[\![u]\!]) \subset N[\![u ]\!] \otimes M\]
and the specialization $(u^d \wh{R}_{M,N}) |_{u=0} \colon M \otimes N \to N \otimes M$ is non-zero.
\end{Rem}

\subsection{E-invariants}
Let $Q$ be an acyclic quiver.
We denote by $\rep \C Q$ the category of finite-dimensional representations of $Q$ over $\C$ and by $\inj\C Q$ the full subcategory of $\rep\C Q$ consisting of injective representations. 
Let $C^2(\inj \C Q)$ be the category of morphisms $\phi \colon I^{(0)} \to I^{(1)}$ in $\inj\C Q$.
We refer to an object of $C^2(\inj \C Q)$ as an injective copresentation of $Q$.
We regard an object of $C^2(\inj \C Q)$ as a cochain complex concentrated in the cohomological degrees $0$ and $1$. For any $\phi, \psi \in C^2(\inj \C Q)$, the $E$-invariant between them is defined to be the vector space
\[E(\phi, \psi) \seq \Hom_{K^b(\inj \C Q)}(\phi, \psi[1]), \]
where $K^b(\inj\C Q)$ is the homotopy category of bounded complexes in $\inj \C Q$ and $[1]$ is the shift functor.
This is a finite-dimensional $\C$-vector space.
Note that $K^b(\inj\C Q)$ is naturally equivalent to the derived category $D^b(\rep\C Q)$ and that $\phi$ is isomorphic to $\Ker\phi[0] \oplus \Cok\phi[-1]$ in $D^b(\rep\C Q)$. 
Since $\C Q$ is hereditary, quotients of injective modules are injective. 
In particular, $\Cok \phi$ belongs to $\inj \C Q$.
Therefore, we have
\begin{equation} \label{eq:Einv}
E(\phi, \psi) \simeq \Ext^1_Q(\Ker \phi, \Ker\psi) \oplus \Hom_{Q}(\Ker\phi, \Cok\psi).
\end{equation}

\subsection{Main theorem}
\label{Ssec:state}

In what follows, we assume that $\fg$ is of simply-laced type (i.e., type $\mathrm{ADE}$).
An integer-valued function $\xi \colon \II \to \Z$ is called a \textit{height function} if it satisfies $|\xi(i)-\xi(j)| = 1$ whenever $i$ and $j$ are adjacent in the Dynkin diagram of $\fg$.
A height function $\xi$ defines a Dynkin quiver $Q_\xi$ of the same type as $\fg$ in the following way.
The set of vertices of $Q_\xi$ is $\II$.
For an adjacent pair $(i,j)$ in $\II$, we have an arrow $i \to j$ in $Q_\xi$ if $\xi(i) = \xi(j) +1$. 
Note that we have $Q_\xi = Q_{\xi'}$ if and only if the difference $\xi - \xi'$ is constant.
Any Dynkin quiver arises from a height function in this way.

\begin{Ex} \label{Ex:QA3}
When $\fg$ is of type $\mathrm{A}_3$, the function $\xi$ given by $\xi(i)=i$ under the standard identification $\II = \{1,2,3 \}$ is a height function. 
The associated quiver $Q_\xi$ is depicted as $Q_\xi = (\xymatrix{ \overset{1}{\circ} & \ar[l] \overset{2}{\circ} & \ar[l] \overset{3}{\circ}})$.
\end{Ex}

Throughout this paper, we fix a height function $\xi$. 
For each $i \in \II$, 
let $S_i \in \rep \C Q_\xi$ be the simple representation at $i$, and
$I_i \in \inj \C Q_\xi$ an injective hull of $S_i$.
Let $\N \seq \Z_{\ge 0}$. 
For each pair $w = (w(0), w(1)) \in \NII = \N^\II \times \N^\II$ of $I$-tuples of non-negative integers,
we set 
\[I^{w(0)} \seq \bigoplus_{i \in I}I_i^{\oplus w_i(0)}, \quad 
I^{w(1)} \seq \bigoplus_{i \in I}I_i^{\oplus w_i(1)}, \quad 
X(w) \seq \Hom_{Q_\xi}(I^{w(0)}, I^{w(1)}),
\]
where $w(k) = (w_i(k))_{i \in I} \in \N^I$ for $k = 0,1$.
The automorphism group 
\[A(w) \seq \Aut_{Q_\xi}(I^{w(0)}) \times \Aut_{Q_\xi}(I^{w(1)})\] 
acts on the vector space $X(w)$ in the natural way.
Since $Q_\xi$ is of finite representation type, there are only finitely many $A(w)$-orbits in $X(w)$ \cite[Corollary~2.6]{DF}.
In particular, there exists a unique open orbit.

\begin{Def} \label{Def:phiw}
For each $w \in \NII$, we denote by 
\[\phi_\xi(w) \colon I^{w(0)} \to I^{w(1)}\] 
an injective copresentation in the unique open $A(w)$-orbit in $X(w)$.  
By definition, it is unique up to $A(w)$-congugation.
\end{Def}

{We say that an injective copresentation $\phi$ is \textit{rigid} if $E(\phi, \phi) = 0$.
By \cite[p.215]{DF}, $\phi \in X(w)$ is rigid if and only if $\phi \simeq \phi_\xi(w)$ (see also the second paragraph of \S\ref{Ssec:slice} below).
Therefore, the set $\{ \phi_\xi(w) \mid w\in \NII\}$ gives a complete system of rigid injective copresentations of $Q_\xi$. }

On the other hand, to each $w = (w(0), w(1)) \in \NII$, we associate a simple $U_q(L\fg)$-module $L_\xi(w)$ in the category $\Cc$ by
\begin{equation} \label{eq:Lxi}
L_\xi(w) \seq L\left(\prod_{i \in I}\varpi_{i,q^{\xi(i)}}^{w_i(0)} \varpi_{i,q^{\xi(i)+2}}^{w_i(1)}\right).
\end{equation}
\begin{Def}[{Hernandez--Leclerc \cite{HL, HL2}}]
The \textit{level-one subcategory} $\Cc_{\xi,1}$ is defined to be the Serre subcategory of $\Cc$ generated by the simple modules $L_\xi(w)$ for $w \in \NII$.  
\end{Def}
By \cite[Lemma~3.2]{HL2}, the category $\Cc_{\xi, 1}$ is a monoidal subcategory of $\Cc$.

The main theorem of this paper is the following.

\begin{Thm} \label{Thm:main}
For any height function $\xi$ and $w, w' \in \NII$, we have
\[ \oo(L_\xi(w), L_\xi(w')) = \dim E(\phi_\xi(w), \phi_\xi(w')).\]
\end{Thm}
A proof is given later in \S\ref{Ssec:proof}.

 \subsection{Cluster theoretical interpretation}
 \label{Ssec:cluster}
In this subsection, we briefly explain the equivalence between Theorem~\ref{Thm:main1} and Theorem~\ref{Thm:main}.
It amounts to giving a cluster theoretical interpretation of Theorem~\ref{Thm:main}.

\subsubsection{Cluster algebras}
First, we fix our notation around the finite type cluster algebras.
Recall that we have fixed a height function $\xi \colon \II \to \Z$.
Let $\II' \seq \{ i' \mid i \in \II\}$ be a copy of the set $\II$, which serves the set of frozen indices.
Let $A_\xi$ be the cluster algebra of geometric type associated with the exchange matrix $\tB = (b_{ij})_{i \in \II \sqcup \II', j \in \II}$ given by 
\[ b_{ij} \seq n_{ij} - n_{ji}, \quad b_{i'j} = \delta_{i,j} - n_{ij} \]
for $i,j \in \II$,  
where $n_{ij}$ denotes the number of arrows from $i$ to $j$ of the quiver $Q_\xi$, and $\delta_{i,j}$ is the Kronecker delta.
By the Laurent phenomenon, $A_\xi$ is the subring generated by all the cluster variables inside the ring of Laurent polynomials in the initial cluster variables $\{ x_i \mid i \in \II \sqcup \II' \}$. See \cite{FZ1}.

Let $\Delta^+$ denote the set of positive roots of $\fg$ and $\alpha_i \in \Delta^+$ the $i$-th simple root. 
Since $A_\xi$ is of finite type, the set of non-frozen cluster variables of $A_\xi$ is finite and in bijection with the set of almost positive roots $\Delta_{\ge -1} \seq \Delta^+ \cup \{ -\alpha_i \mid i \in \II\}$. See \cite{FZ2}. 
Let $x[\alpha]$ denote the cluster variable corresponding to $\alpha \in \Delta_{\ge -1}$.
For example, we have $x[-\alpha_i] = x_i$ and $x[\alpha_i] = x^{-1}_i(\prod_{j \in \II} x_j^{n_{ij}}x_{j'}^{n_{ji}} + x_{i'} \prod_{j \in \II}x_j^{n_{ji}})$ for each $i \in \II$.
For a positive root $\alpha = \sum_{i \in \II}a_i \alpha_i$,  the cluster variable $x[\alpha]$ is the one having $\prod_{i \in \II}x_i^{a_i}$ as its denominator.
The cluster variables $x[\alpha] \ (\alpha \in \Delta_{\ge -1})$, $x_{i'} \ ( i \in \II)$ are grouped into several subsets of constant cardinality $2 |\II|$, called the clusters.
A cluster always contains the frozen variables $\{ x_{i'} \mid i \in \II\}$. 
A monomial of cluster variables of a common cluster is called a cluster monomial.     
The cluster monomials form a {$\Z$-linearly independent subset} of $A_\xi$, and equivalently, the cluster monomials without frozen factors form a {$\Z[x_{i'} \mid i \in \II]$-linearly independent subset} of $A_\xi$ {by \cite[Theorem 11.2]{FZ4}}.  

\subsubsection{Additive categorification}
\label{Sssec:decr}
For $M \in \rep \C Q_\xi$, its dimension vector is $\vdim M \seq \sum_{i \in \II} (\dim M_i) \alpha_i$.  
By Gabriel's theorem, for each $\alpha \in \Delta^+$, there exists an indecomposable object $M_\xi[\alpha] \in \rep \C Q_\xi$ uniquely up to isomorphism satisfying $\vdim M_\xi[\alpha] = \alpha$, and the set $\{ M_\xi[\alpha] \mid \alpha \in \Delta_{\ge -1}\}$ gives a complete system of indecomposable objects of $\rep \C Q_\xi$.
For $v = (v_i)_{i \in \II} \in \N^\II$, we set $\C^v \seq \bigoplus_{i \in \II} \C^{v_i}$.
Recall that a decorated representation of $Q_\xi$ is a pair $\mathcal{M} = (M,V)$ of $M \in \rep \C Q_\xi$ and a finite-dimensional $\II$-graded $\C$-vector space $V$.
We set $\cM_\xi[\alpha] \seq (M_\xi[\alpha], 0)$ for $\alpha \in \Delta^+$ and $\cM_\xi[-\alpha_i] \seq (0,\C^{\delta_i})$ for $i \in \II$, where $\delta_i = (\delta_{i,j})_{j \in \II} \in \N^\II$ is the delta function at $i$.  
Then, the set $\{\cM_\xi[\alpha] \mid \alpha \in \Delta_{\ge -1} \}$ gives a complete system of indecomposable decorated representations of $Q_\xi$.

For a decorated representation $\cM = (M,V)$ of $Q_\xi$, its cluster character $CC(\cM)$ is defined as in \cite{DWZ}:
\[ CC(\cM) \seq \sum_{v \in \N^\II} \chi(\Gr_v(M)) {\prod_{i \in \II \sqcup \II'} x_{i}^{\tilde{g}_i(\cM)-\sum_{j \in \II}b_{ij} v_j},} \]
where $\chi(\Gr_v(M))$ is the Euler characteristic of the submodule Grassmannian $\Gr_v(M)$, i.e., the complex projective variety  parametrizing subrepresentations of $M$ of dimension vector $\sum_{i \in \II}v_i \alpha_i$, and $(\tilde{g}_i(\cM))_{i \in \II \sqcup \II'}$ is the so-called extended $g$-vector of $\cM$. 
In our case, it is explicitly written as
\[ \tilde{g}_i(\cM) = \dim V_i - \dim M_i + \sum_{j \in \II} n_{ij} \dim M_j, \quad \tilde{g}_{i'}(\cM) = \dim\bigcap_{a}\Ker(a|_M) \]
for each $i \in \II$, where $a$ runs over the set of arrows of $Q_\xi$ whose source is $i$. 
By \cite{CC, DWZ}, we have $CC(\cM \oplus \cM') = CC(\cM) \cdot CC(\cM')$ for any decorated representations $\cM, \cM'$, and $CC(\cM_\xi[\alpha]) = x[\alpha]$ for all $\alpha \in \Delta_{\ge -1}$.
Recall the $E$-invariant for decorated representations defined in \eqref{eq:Edecr}. 
By \cite{MRZ}, two cluster variables $x[\alpha]$ and $x[\alpha']$ belong to a common cluster if and only if we have $E(\cM_\xi[\alpha], \cM_\xi[\alpha']) = E(\cM_\xi[\alpha'], \cM_\xi[\alpha]) =0$.
The map $CC$ gives a bijection between the isomorphism classes of rigid decorated representations of $Q_\xi$ and the cluster monomials without frozen factors of $A_\xi$.

\begin{Rem} \label{Rem:order}
The following remark is used later in \S\ref{Ssec:proof}. 
It is well known that there is a partial ordering $\le_\xi$ of the set $\Delta^+$ with the following property: we have $\alpha \le_\xi \alpha'$ if $\Ext^1_{Q_\xi}(M_\xi[\alpha], M_\xi[\alpha']) \neq 0$. 
We can extend it to the set $\Delta_{\ge -1}$ so that we have $\alpha \le_\xi - \alpha_i$ for any $\alpha \in \Delta^+$ and $i \in \II$. 
Then,  for $\alpha, \alpha' \in \Delta_{\ge -1}$, we have $\alpha \le_\xi \alpha'$ whenever $E(\cM_\xi[\alpha], \cM_\xi[\alpha']) \neq 0$.   
\end{Rem}

\subsubsection{Interpretation by injective copresentations}
\label{Sssec:injcopres}
To each $\phi \in C^2(\inj \C Q_\xi)$, we assign a decorated representation $\cM(\phi)$ of $Q_\xi$ by 
\[ \cM(\phi) \seq (\Ker \phi, \C^c), \quad \text{where $c_i \seq \dim \Hom_{Q_\xi}(S_i, \Cok \phi)$}.\]
Comparing \eqref{eq:Edecr} and \eqref{eq:Einv}, we find that 
\begin{equation} \label{eq:EE}
E(\phi, \psi) \simeq E(\cM(\phi), \cM(\psi))
\end{equation} holds for any $\phi, \psi \in C^2(\inj \C Q_\xi)$.
For each $\alpha \in \Delta^+$, let $\phi_\xi[\alpha]$ be a minimal injective resolution of $M_\xi[\alpha]$.
For each $i \in \II$, we set $\phi_\xi[-\alpha_i] \seq (0 \to I_i)$ and $\nu_i \seq (I_i \xrightarrow{\id} I_i)$.
By construction, we have $\cM(\phi_\xi[\alpha]) = \cM_\xi[\alpha]$ for any $\alpha \in \Delta_{\ge -1}$ and $\cM(\nu_i) = 0$ for any $i \in \II$. 

The set $\{ \phi_\xi[\alpha] \mid \alpha \in \Delta_{\ge -1}\} \sqcup \{\nu_i \mid i \in \II \}$ forms a complete system of indecomposable objects of $C^2(\inj \C Q_\xi)$.
Since the category $C^2(\inj \C Q_\xi)$ is Krull--Schmidt, each object decomposes into a finite direct sum of indecomposable objects in a unique way. 
For each $\phi \in C^2(\inj \C Q_\xi)$, we define \begin{equation}
CC(\phi) \seq CC(\cM(\phi)) \prod_{i \in \II} x_{i'}^{m_i(\phi)},
\end{equation} 
where $m_i(\phi)$ denotes the multiplicity of the factor $\nu_i$ in $\phi$.

\begin{Lem}
The map $w \mapsto CC(\phi_\xi(w))$ gives a bijection from $\NII$ to the set of cluster monomials of $A_\xi$.
\end{Lem}
\begin{proof}
{As mentioned in \S\ref{Ssec:state} above,} the set $\{ \phi_\xi(w) \mid w\in \NII\}$ gives a complete system of rigid injective copresentations of $Q_\xi$. 
On the other hand, we know that $\phi$ is rigid if and only if $\cM(\phi)$ is rigid by \eqref{eq:EE}, and that $\cM(\phi) \simeq \cM(\psi)$ if and only if $[\phi] - [\psi] \in \sum_{i \in \II} \Z [\nu_i]$ in the Grothendieck group of $C^2(\inj \C Q_\xi)$.
Having these remarks, the assertion now follows from the results explained in the second paragraph of \S\ref{Sssec:decr}.    
\end{proof}

\subsubsection{Monoidal categorification}
The following theorem was originally conjectured by Hernandez--Leclerc \cite{HL} when $Q_\xi$ has a sink-source orientation.
For this case, it was proved by Hernandez--Leclerc \cite{HL} for type $\mathrm{A}$, $\mathrm{D}_4$, and by Nakajima \cite{NakCA} for all type $\mathrm{ADE}$.
For general $\xi$, some results were obtained by Hernandez--Leclerc \cite{HL2} and Brito--Chari \cite{BC}. 
In full generality, it was proved by Kimura--Qin \cite{KQ}. 

\begin{Thm} \label{Thm:mcat}
There is a ring isomorphism $\tchi_q \colon K(\Cc_{\xi, 1}) \xrightarrow{\simeq} A_\xi$ satisfying 
\[ \tchi_q(L_\xi(w)) = CC(\phi_\xi(w)) \]
for all $w \in \NII$. 
In particular, $\tchi_q$
induces a bijection between the simple isomorphism classes of $\Cc_{\xi,1}$ and the cluster monomials of $A_\xi$.
\end{Thm}

\subsubsection{Conclusion}
Now, it is easy to see that Theorem~\ref{Thm:main1} and Theorem~\ref{Thm:main} are mutually equivalent by Theorem~\ref{Thm:mcat} and \eqref{eq:EE}.

\section{Geometric preliminaries}
\label{Sec:geom}
In this section, we give a brief review of the geometric construction of finite-dimensional representations of $U_q(\Lfg)$ by means of equivariant constructible sheaves on the graded quiver varieties due to Nakajima.
Basic references are \cite{Nak, NakT} {(see also \cite{Nak1, Nak2})}. 
There are no new results in this section.
 
\subsection{Notation}
Let $\gV_\C$ denote the category of $\Z$-graded $\C$-vector spaces $V = \bigoplus_{k \in \Z}V^k$ of finite total dimension, i.e., $\sum_{n \in \Z}\dim V^n < \infty$, whose morphisms are homogeneous linear maps.
Let $t$ be an indeterminate.
For $V \in \gV_\C$, its graded dimension is defined to be
\[ \gdim (V) \seq \sum_{n \in \Z} (\dim V^n)t^n.\]
This is an element of $\N[t^{\pm 1}]$.
For $V, W \in \gV_\C$ and $l \in \Z$, we denote by $\Hom^l(V,W)$ the space of $\C$-linear maps $f \colon V \to W$ of degree $l$, i.e., satisfying $f(V^n) \subset W^{n+l}$ for all $n \in \Z$.
Let 
\[G(V) \seq \Hom^{0}(V,V)^\times = \prod_{n \in \Z}GL(V^n).\] 

In what follows, a variety always means a complex algebraic variety.
When a complex algebraic group $G$ acts on a variety $X$, we say that $X$ is a $G$-variety.
We set $\pt \seq \Spec \C$ and view it as a $G$-variety with the trivial action. 
Given a field $\kk$ and a $G$-variety $X$, we denote by $D^b_G(X,\kk)$ the bounded $G$-equivariant derived category of constructible $\kk$-complexes on $X$ in the sense of Bernstein--Lunts \cite{BL} (see also \cite[Ch.~6]{Ach}).
This is a triangulated category equipped with a $t$-structure whose heart is identical to the category $\Perv_G(X, \kk)$ of $G$-equivariant perverse sheaves on $X$. 
For any objects $\mathcal{F}$, $\mathcal{G} \in D^b_G(X,\kk)$, we set
\[ \Hom^\bullet_G(\mathcal{F,G}) \seq \bigoplus_{n \in \Z}\Hom^n_G(\mathcal{F,G}), \quad 
\Hom^\bullet_G(\mathcal{F,G})^\wedge \seq \prod_{n \in \Z}\Hom^n_G(\mathcal{F,G})\]
where $\Hom^n_G(\mathcal{F, G}) \seq \Hom_{D_G^b(X,\kk)}(\mathcal{F,G}[n])$ with $[1]$ being the shift functor. 
We also use the notations $\rH^\bullet_G(\mathcal{F}) \seq \Hom^\bullet_G(\uk_X, \mathcal{F})$ and $\hrH^\bullet_G(\mathcal{F}) \seq \Hom^\bullet_G(\uk_X, \mathcal{F})^\wedge$, where $\uk_X$ is the constant $\kk$-sheaf (of rank one).
When $\mathcal{F} = \uk_X$, we recover the $G$-equivariant cohomology ring of $X$ and hence write $\rH_G^\bullet(X,\kk) = \rH^\bullet_G(\uk_X)$ and $\hrH_G^\bullet(X, \kk) \seq \hrH^\bullet_G(\uk_X)$.

For a group homomorphism $f \colon H \to G$, we have the associated functor $\Res_f \colon D^b_G(X, \kk) \to D^b_H(X,\kk)$ of equivariance change. 
When $f$ is the inclusion of a subgroup or the quotient by a normal subgroup, we denote it by $\For^G_H$ or $\Infl_G^H$ respectively.
When $f$ is understood from the context, we often drop $\Res_f$ from the notation for the sake of simplicity.
For instance, we often denote
$\Hom^\bullet_H(\Res_f(\mathcal{F}), \Res_f(\mathcal{G}))$ and $\rH^\bullet_H(\Res_f(\mathcal{F}))$ simply by $\Hom^\bullet_H(\mathcal{F, G})$ and $\rH^\bullet_H(\mathcal{F})$ respectively. 
The following fact is used several times below.

\begin{Lem}[{\cite[Lemma~6.7.4]{Ach}}] \label{Lem:eq_change}
Let $f \colon H \to G$ be a group homomorphism and $\mathcal{F} \in D^b_G(X, \kk)$.
If $\rH^\bullet_G(\mathcal{F})$ is free over $\rH^\bullet_G(\pt, \kk)$, we have
\[ \rH_H^\bullet(\mathcal{F}) \simeq \rH^\bullet_G(\mathcal{F}) \otimes_{\rH^\bullet_G(\pt, \kk)} \rH^\bullet_H(\pt, \kk),\]
where $\rH^\bullet_G(\pt, \kk) \to \rH^\bullet_H(\pt, \kk)$ is induced from $f$.
\end{Lem}

\subsection{Graded quiver varieties}
For $V = (V_i)_{i \in \II}$, $W = (W_i)_{i \in \II} \in (\gV_\C)^\II$, we consider the space of linear maps
\[ \mathrm{M}^\bullet(V, W) \seq \bigoplus_{i, j \in \II, c_{ij}<0} \Hom^{-1}(V_i, V_j)  \oplus \bigoplus_{i \in \II} \left(\Hom^{-1}(V_i,W_i) \oplus \Hom^{-1}(W_i,V_i)\right),  
\]
where $(c_{ij})_{i,j \in \II}$ is the Cartan matrix of $\fg$.
The groups
\[ G(V) \seq \prod_{i \in \II} G(V_i), \quad G(W) \seq \prod_{i \in \II}G(W_i) \]
act on $\mathrm{M}^\bullet(V, W)$ by conjugation. 
Let $\mu \colon \mathrm{M}^\bullet(V,W) \to \bigoplus_{i \in \II}\Hom^{-2}(V_i,V_i)$ be the $G(V)$-equivariant map given by
\[\mu((B_{j,i}), (a_i), (b_i)) \seq (\textstyle\sum_{j \in \II, c_{ij}<0}B_{i,j} B_{j,i} + b_ia_i)_{i \in \II},\]
where $B_{j,i} \in \Hom^{-1}(V_i, V_j)$, $a_i \in \Hom^{-1}(V_i, W_i)$, and $b_i \in \Hom^{-1}(W_i, V_i)$.
We say that a point $((B_{j,i}), (a_i), (b_i)) \in \mu^{-1}(0)$ is stable if there is no non-zero $\Z$-graded linear subspace $V'_i \subset V_i$ for any $i \in \II$ such that $B_{j,i}(V'_i) = 0$ for any $j \in \II$ with $c_{ij}<0$.
The group $G(V)$ acts freely on the (possibly empty) open subset $\mu^{-1}(0)^{st} \subset \mu^{-1}(0)$ of stable points. 
The quotient 
\[\qv(V,W) \seq \mu^{-1}(0)^{st}/G(V) \]
is a smooth quasi-projective $G(W)$-variety.
It can be identified with a quotient in the geometric invariant theory.
In particular, it comes with a natural projective $G(W)$-equivariant morphism
\begin{equation} \label{eq:pivw}
\pi_{V, W} \colon \qv(V,W) \to \qv_0(V,W) \seq \Spec \C[\mu^{-1}(0)]^{G(V)}.
\end{equation} 
Both varieties $\qv(V,W)$ and $\qv_0(V,W)$ only depend on the graded dimension vector of $V$.
Therefore, we write $\qv(\bv,W) \seq \qv(V,W)$, $\qv_0(\bv, W) \seq \qv_0(V,W)$, and $\pi_{\bv, W} \seq \pi_{V,W}$ when $\bv = (\gdim (V_i))_{i \in \II} \in \N[t^{\pm 1}]^\II$.

A geometric point of the affine variety $\qv_0(\bv, W)$ corresponds to a closed $G(V)$-orbit in $\mu^{-1}(0)$.
Let $\qv_0(\bv, W)^{reg}$ be the smooth open subvariety of $ \qv_0(\bv, W)$ corresponding to free orbits. 
It is non-empty if and only if
\begin{equation} \label{eq:ldom}
\qv(\bv,W) \neq \varnothing \quad \text{ and } \quad
(\gdim(W_i))_{i \in \II} - C(t) \cdot \bv \in \N[t^{\pm 1}]^\II, 
\end{equation}
where $C(t) \seq (\frac{t^{c_{ij}}-t^{-c_{ij}}}{t-t^{-1}})_{i,j \in \II}$ is the quantum Cartan matrix.
The set
\[ \Dom(W) \seq \{ \bv \in \N[t^{\pm 1}]^\II \mid \text{the condition~\eqref{eq:ldom} is satisfied}\}\]
is finite.
For $\bv \in \Dom(W)$, the morphism $\pi_{\bv,W}$ restricts to an isomorphism 
\begin{equation} \label{eq:isomreg}
\pi^{-1}_{\bv,W}(\qv_0(\bv, W)^{reg}) \xrightarrow{\simeq} \qv_0(\bv, W)^{reg}.
\end{equation}

For $\bv, \bv' \in \N[t^{\pm 1}]^\II$, we have a natural closed embedding $\qv_0(\bv, W) \subset \qv_0(\bv+\bv', W)$.
Taking the unions over $\bv \in \N[t^{\pm 1}]^\II$, we define
\[ \qv(W) \seq \bigsqcup_{\bv} \qv(\bv, W), \quad  \qv_0(W) \seq \bigcup_{\bv} \qv_0(\bv, W).\]
These unions are essentially finite as $\qv(\bv, W) \neq \varnothing$ only for finitely many $\bv$ and $\qv_0(\bv, W)$ stabilizes  for sufficiently large $\bv$.
The morphisms~\eqref{eq:pivw} are unified into a $G(W)$-equivariant projective morphism 
\[ \pi_{W} \colon \qv(W) \to \qv_0(W).\]
The locally closed subvarieties $\{\qv_0(\bv, W)^{reg}\}_{\bv \in \Dom(W)}$ give a finite stratification of $\qv_0(W)$.
Note that $\qv_0(0, W)^{reg} = \{ 0 \}$ is the unique closed stratum.

In what follows, we assume that $\kk$ is an algebraically closed field containing $\Q(q)$ as in \S\ref
{Ssec:qloop}.
Consider the proper push-forward $(\pi_{W})_! \uk_{\qv(W)}$ of the constant $\kk$-sheaf on $\qv(W)$ and let $\cL'_W \seq {}^p\rH^\bullet((\pi_{W})_! \uk_{\qv(W)})$ denote its total perverse cohomology.
By the decomposition theorem, this is a semisimple object in $\Perv_{G(W)}(\qv_0(W), \kk)$.
More precisely, we have 
\begin{equation} \label{eq:DT}
 \cL'_W  \simeq \bigoplus_{\bv \in \Dom(W)} L_{\bv, W} \boxtimes \IC_{\bv, W},
 \end{equation}
where $\IC_{\bv, W} \in D^b_{G(W)}(\qv_0(W), \kk)$ is the intersection cohomology complex of $\ol{\qv_0(\bv, W)^{reg}}$ and  $L_{\bv, W}$ is a non-zero finite-dimensional $\kk$-vector space.
Note that $\IC_{0,W}$ is the skyscraper sheaf $\uk_{\{0\}}$ at the origin $0 \in \qv_0(W)$.

\subsection{Nakajima's homomorphism}
For each $W \in (\gV_\C)^\II$, we consider the completed Yoneda algebra
$\Hom_{G(W)}^\bullet(\cL'_W, \cL'_W)^\wedge$.
Note that this is the completion of an $\N$-graded algebra $\Hom_{G(W)}^\bullet(\cL'_W, \cL'_W)$ with a semisimple $0$-th component $\Hom_{G(W)}^0(\cL'_W, \cL'_W) \simeq \prod_{\bv \in \Dom(W)}\End_\kk(L_{\bv, W})$. 
In particular, each $L_{\bv, W}$ can be regarded as a simple module over $\Hom_{G(W)}^\bullet(\cL'_W, \cL'_W)^\wedge$ and the set $\{ L_{\bv, W} \}_{\bv \in \Dom(W)}$ gives a complete system of simple modules.

\begin{Thm}[Nakajima \cite{Nak}]
\label{Thm:Nak}
There is a homomorphism of $\kk$-algebras
\[\varphi'_W \colon U_q(\Lfg) \to \Hom_{G(W)}^\bullet(\cL'_W, \cL'_W)^\wedge\]
satisfying the following property.
For any $\bv \in \Dom(W)$, the pullback $(\varphi'_W)^* L_{\bv, W}$ is a simple $U_q(\Lfg)$-module in $\Cc$ isomorphic to $L(\prod_{i \in \II, n \in \Z}\varpi_{i,q^n}^{m_{i,n}})$, where the multiplicities $m_{i,n} \in \N$ are determined by the formula
\[ \left(\textstyle\sum_{n \in \Z} m_{i,n} t^n\right)_{i \in \II} = (\gdim(W_i))_{i \in \II} - C(t) \cdot \bv.\] 
In particular, when $\bv = 0$, we have $(\varphi'_W)^*L_{0,W} \simeq L(\varpi_W)$, where 
\begin{equation} \label{eq:varpiW}
\varpi_W \seq \prod_{i \in \II, n \in \Z} \varpi_{i,q^n}^{\dim W_i^n}.
\end{equation}
\end{Thm}
\begin{proof}
The $\kk$-algebra homomorphism $\varphi_W$ is obtained as the composition of (i) a completion of the homomorphism in \cite[Theorem~9.4.1]{Nak} from $U_q(\Lfg)$ to the convolution algebra $\widehat{K}^{G(W)}(Z^\bullet(W))_\kk$ of the completed $G(W)$-equivariant $K$-theory (see \cite[\S4.6]{Frmat} for details), where $Z^\bullet(W) \seq \qv(W) \times_{\qv_0(W)} \qv(W)$ is the Steinberg type variety, (ii) the $G(W)$-equivariant Chern character map (suitably twisted by the Todd classes) from $\wh{K}^{G(W)}(Z^\bullet(W))_\kk$ to the convolution algebra $\wh{\rH}^{G(W)}_\bullet(Z^\bullet(W), \kk)$ of the completed $G(W)$-equivariant Borel--Moore homology (equivariant version of \cite[Theorem~5.11.11]{CG}), and (iii) the completion of an isomorphism between $\rH^{G(W)}_\bullet(Z^\bullet(W), \kk)$ and $\Hom_{G(W)}^\bullet(\cL'_W, \cL'_W)$ (equivariant version of the isomorphism in \cite[\S8.6]{CG}).
The desired property is due to \cite[Theorem~14.3.2]{Nak}.  
\end{proof}

\subsection{Deformed standard modules}
For each $\varpi \in (1 + z\kk[z])^\II$, the \textit{standard module} (also known as the \textit{local Weyl module} in the sense of Chari--Pressley \cite{CPweyl}) $M(\varpi)$ is defined.
It is the largest $\ell$-highest weight module in $\Cc$ and it has $L(\varpi)$ as a unique simple quotient.

Fix $W \in (\gV_\C)^\II$. 
Let $T(W) = \C^\times \id_W$ denote the one-dimensional torus of non-zero scalar matrices in $G(W)$ and
$i_0 \colon \{ 0\} \to \qv_0(W)$ the inclusion.
The action of $T(W)$ on $\qv_0(W)$ is trivial.
Through the Yoneda product and the functor $\For^{G(W)}_{T(W)}$, 
the algebra $\Hom^\bullet_{G(W)}(\cL'_W, \cL'_W)^\wedge$ acts on $\hrH_{T(W)}^\bullet(i_0^! \cL'_W)$.
Via $\varphi'_W$, this yields a geometric realization of the deformed standard module $M(\varpi_W)[\![u]\!]$ (recall the definition of $M[\![u]\!]$ from Remark~\ref{Rem:Mu} and $\varpi_W$ from \eqref{eq:varpiW}) as follows.

\begin{Thm}[Nakajima \cite{NakExt}] \label{Thm:standard}
We have
\[ (\varphi'_W)^*\hrH^\bullet_{T(W)}(i_0^!\cL'_W) \simeq M(\varpi_W)[\![ u ]\!]\]
as $U_q(\Lfg)[\![u]\!]$-modules, where the action of $u$ on the left-hand side is given by  the product with a non-zero element of $\rH^2_{T(W)}(\pt, \kk)$.
\end{Thm}
\begin{proof}
This follows from \cite[Theorem~2 and Remark~2.15]{NakExt} and a completion.
{To be more precise, we apply the same completion as in the proof of Theorem \ref{Thm:Nak} to find that the $U_q(L\fg)$-module $(\varphi'_W)^*\hrH^{G(W)}_\bullet(\pi^{-1}(0), \kk)$ is isomorphic to the completion $\widehat{M}(\varpi_W)$ of the global Weyl module of highest weight $\varpi_W \seq \sum_{i \in \II} \dim(W_i)\varpi_i$ (with $\varpi_i$ being the $i$-th fundamental weight for $\fg$) at the maximal ideal of its endomorphism ring corresponding to $\varpi_W$.
The module $\widehat{M}(\varpi_W)$ is the same thing as what the author called the deformed local Weyl module in \cite[\S2.3]{Fahw}. 
The endomorphism ring of $\widehat{M}(\varpi_W)$ is identical to $\hrH^\bullet_{G(W)}(\pt, \kk)$ and we have 
\[M(\varpi_W)[\![ u ]\!] \simeq \widehat{M}(\varpi_W) \otimes_{\hrH^\bullet_{G(W)}(\pt, \kk)}\hrH^\bullet_{T(W)}(\pt, \kk).\]
}
Note that $\rH^\bullet_{G(W)}(i_0^!\cL'_W) \simeq \rH^{G(W)}_\bullet(\pi^{-1}(0), \kk)$ is free over $\rH^\bullet_{G(W)}(\pt, \kk)$ by \cite[\S7.1]{Nak} and hence Lemma~\ref{Lem:eq_change} is applicable.
{Therefore, we have a natural isomorphism
\[ \hrH^{G(W)}_\bullet(\pi^{-1}(0), \kk) \otimes_{\hrH^\bullet_{G(W)}(\pt, \kk)}\hrH^\bullet_{T(W)}(\pt, \kk) \simeq \hrH^\bullet_{T(W)}(i_0^!\cL'_W). \]
The left-hand side
is isomorphic to $M(\varpi_W)[\![ u ]\!]$ via $\varphi'_W$.}
\end{proof}

\subsection{Tensor product}
Let $W, W' \in (\gV_\C)^\II$.
We identify the one-dimensional torus $T(W') \subset G(W')$ with the subtorus $\id_W \oplus \C^\times \id_{W'}$ of $G(W \oplus W')$. 
By \cite[Lemma~3.1]{VV}, the $T(W')$-fixed locus $\qv_0(W\oplus W')^{T(W')}$ is identical to $\qv_0(W) \times \qv_0(W')$.
Consider the attracting locus
\[ \fA^{\pm}(W, W') \seq \{ x \in \qv_0(W \oplus W') \mid \lim_{s \to 0} (\id_W \oplus s^{\pm 1} \id_{W'}) x \text{ exists.}\}, \]
which is Zariski closed by \cite[3.5]{VV}. 
We have $\fA^{\pm}(W, W') = \fA^{\mp}(W', W)$.
Consider the diagram
\begin{equation} \label{eq:diag_res}
\qv_0(W) \times \qv_0(W')\xleftarrow{p'_\pm} \fA^\pm(W, W') \xrightarrow{h'_\pm} \qv_0(W\oplus W'),
\end{equation}
where $h'_\pm$ is the inclusion and $p'_\pm(x) \seq \lim_{s \to 0} (\id_W \oplus s^{\pm 1} \id_{W'}) x${. Recall} the hyperbolic localization $(p'_{\pm})_! (h'_\pm)^* \simeq (p'_\mp)_* (j'_\mp)^!$ in the sense of Braden~\cite{Bra}.
By \cite[Lemma~4.1]{VV}, $(p'_{\pm})_! (h'_\pm)^* \cL'_{W \oplus W'}$ is a semisimple complex and
\begin{equation} \label{eq:VVisom}
{}^p \rH^\bullet((p'_{\pm})_! (h'_\pm)^* \cL'_{W \oplus W'}) \simeq \cL'_W \boxtimes \cL'_{W'}
 \end{equation}
 in $\Perv_{G(W) \times G(W')}(\qv_0(W) \times \qv_0(W'))$.
The isomorphism~\eqref{eq:VVisom} (together with \cite[Proposition 6.7.5]{Ach}) yields isomorphisms
\begin{align*} 
\hrH^\bullet_{G(W) \times G(W')}(i_0^!(p'_+)_! (h'_+)^*\cL'_{W\oplus W'}) &\simeq \hrH^\bullet_{G(W)}(i_0^!\cL'_{W}) \wh{\otimes} \hrH^\bullet_{G(W')}(i_0^!\cL'_{W'}),  \\
\hrH^\bullet_{G(W) \times G(W')}(i_0^!(p'_-)_! (h'_-)^*\cL'_{W\oplus W'}) &\simeq \hrH^\bullet_{G(W')}(i_0^!\cL'_{W'}) \wh{\otimes} \hrH^\bullet_{G(W)}(i_0^!\cL'_{W}), 
\end{align*}
where $\wh{\otimes}$ denotes the completed tensor product and $i_0$ denotes the inclusions of the origin (into suitable varieties).
A sheaf-theoretic interpretation of the results from \cite[{\S6}]{NakT} tells us that these isomorphisms are compatible with the structures of $U_q(\Lfg)$-modules, given through the homomorphism $\varphi'_{W \oplus W'}$ on the left-hand sides, and through $(\varphi'_W \otimes \varphi'_{W'}) \circ \Delta$ and $(\varphi'_{W'} \otimes \varphi'_{W}) \circ \Delta$ respectively on the right hand sides, where $\Delta$ is the coproduct of $U_q(\Lfg)$.
In particular, applying $\For_{T(W')}^{G(W) \times G(W')}$, we get the following from Theorem~\ref{Thm:standard}. 
(We can freely use Lemma~\ref{Lem:eq_change} here, as the freeness assumption is satisfied, see \cite[\S7.1]{Nak}, \cite[Theorem~3.10(1)]{NakT}.)
 
\begin{Thm}[Nakajima \cite{NakT}] 
\label{Thm:NakT}
We have 
\begin{align*} 
(\varphi'_{W \oplus W'})^* \hrH^\bullet_{T(W')}(i_0^!(p'_+)_! (h'_+)^*\cL'_{W\oplus W'}) & \simeq M(\varpi_W) \otimes (M(\varpi_{W'})[\![u]\!]), \\
(\varphi'_{W \oplus W'})^* \hrH^\bullet_{T(W')}(i_0^!(p'_-)_!(h'_-)^* \cL'_{W \oplus W'}) & \simeq (M(\varpi_{W'})[\![u]\!]) \otimes M(\varpi_W)
\end{align*}
as $U_q(\Lfg)[\![u]\!]$-modules, where the action of $u$ on the left-hand sides is given by  the product with a non-zero element of $\rH^2_{T(W')}(\pt, \kk)$.\end{Thm}

\section{Proof of main theorem}
\label{Sec:proof}

In this section, we prove our main theorem (= Theorem~\ref{Thm:main}) applying the geometric construction reviewed in the previous section. 
In \S\ref{Ssec:NKQ}, we recall the key observation due to Kimura--Qin~\cite{KQ}, which enables us to translate the constructions with equivariant perverse sheaves on the graded quiver varieties to those on the space $X(w)$ of injective copresentations of the Dynkin quiver $Q_\xi$ through the Fourier--Laumon transformation explained in \S\ref{Ssec:FL}. 
Then, we obtain a sheaf theoretic interpretation of deformed simple modules and their tensor products in \S\S\ref{Ssec:LoverX}--\ref{Ssec:LLoverX}. 
In \S\ref{Ssec:slice}, we observe that the $E$-invariant appears as a transversal slice in $X(w)$. 
The proof ends in \S\ref{Ssec:proof}, where we have a sheaf theoretic interpretation of $R$-matrices in question under a certain condition~\eqref{eq:AssumpE}.

\subsection{Graded quiver varieties for $\Cc_{\xi,1}$} \label{Ssec:NKQ}
Recall that we have fixed a height function $\xi \colon \II \to \Z$ and the notations from \S\ref{Ssec:state}. 
For each vertex $i \in \II$, let $P_i \in \rep(\C Q_\xi)$ be a projective cover of $S_i$.
For each $w \in \NII$, we set 
\[ X'(w) \seq \Hom_{Q_\xi}(P^{w(1)}, I^{w(0)}),\]
where $P^{w(1)} \seq \bigoplus_{i \in \II}P_i^{\oplus w_i(1)}$.
Note that the vector spaces $X(w)$ and $X'(w)$ are dual to each other through the Nakayama functor {(cf.\ \cite[III.2.8]{ASS})}.
Moreover, the group $\Aut_{Q_\xi}(P^{w(1)})$ is naturally identified with $\Aut_{Q_\xi}(I^{w(1)})$ and hence the group $A(w)$ defined in \S\ref{Ssec:state} acts on $X'(w)$ as well. 
For $v \in \N^\II$, let $\Gr_v(I^{w(0)})$ denote the submodule Grassmannian of $I^{w(0)}$, 
which is a smooth connected projective variety by \cite[Theorem~4.10]{Rei}.
Consider the $A(w)$-variety
\[ F(v,w) \seq \{ (N, \psi) \in \Gr_v(I^{w(0)}) \times X'(w) \mid \Image \psi \subset N \},\]
together with the $A(w)$-equivariant projective morphism
\[ p_{v,w} \colon F(v,w) \to X'(w) \]
given by the second projection $(N, \psi) \mapsto \psi$. 

Given $w \in \NII$, we choose an  $\II$-tuple of $\Z$-graded vector spaces $W_\xi(w) =(W_\xi(w)_i)_{i \in \II} \in (\gV_\C)^\II$ satisfying
\[\gdim (W_\xi(w)_i) = w_i(0) t^{\xi(i)} + w_i(1) t^{\xi(i)+2}\]
for all $i \in \II$.
Note that $\Dom(W_\xi(w))  \subset  (\N t^{\xi(i)+1})^{\II}$ holds.
For $v = (v_i)_{i \in \II} \in \N^\II$, we put $vt^{\xi+1} \seq (v_i t^{\xi(i)+1})_{i \in \II}$.

The following observation due to Kimura--Qin \cite{KQ}, generalizing the one by Nakajima~\cite{NakCA}, is of fundamental importance in our discussion below.

\begin{Prop}[{Kimura--Qin \cite[Propositions~3.1.1 \& 3.1.4]{KQ}}] \label{Prop:NKQ}
For any $w \in \NII$ and $v \in \N^\II$, we have isomorphisms of varieties
\[ \qv_0(W_\xi(w)) \simeq X'(w), \quad \qv( vt^{\xi+1}, W_\xi(w)) \simeq F(v, w), \]
through which the morphism $\pi_{vt^{\xi+1}, W_\xi(w)}$  corresponds to the morphism $p_{v,w}$, and the actions of $G(W_\xi(w))$ correspond to the actions of the standard Levi subgroup $G(w) \seq \prod_{i \in \II, k\in \{0,1\}} \Aut_{Q_\xi}(I_i^{\oplus w_i(k)})$ of $A(w)$.
\end{Prop}

In what follows, we identify the variety $\qv_0(W_\xi(w))$ with the variety $X'(w)$ through the isomorphism in Proposition~\ref{Prop:NKQ}, and identify $G(w)$ with $G(W_\xi(w))$.
Note that the functor $\For_{G(w)}^{A(w)} \colon D^b_{A(w)}(X,\kk) \to D^b_{G(w)}(X, \kk)$ is fully faithful for any $A(w)$-variety $X$ (cf.~ \cite[Theorem~6.6.15]{Ach}).

\begin{Cor}
For any $w \in \NII$, the object $\cL'_{W_\xi(w)}$ is in the essential image of the functor $\For_{G(w)}^{A(w)} \colon \Perv_{A(w)}(X'(w),\kk) \to \Perv_{G(w)}(X'(w), \kk)$.
\end{Cor}
\begin{proof}
Recall the decomposition~\eqref{eq:DT}.
For any $vt^{\xi + 1} \in \Dom(W_\xi(w))$, we know that the simple perverse sheaf $\IC_{vt^{\xi+1}, W_\xi(w)}$ appears as a direct summand of a shift of $(p_{v, w})_! \uk_{F(v,w)} \in D^b_{A(w)}(X'(w),\kk)$ thanks to Proposition~\ref{Prop:NKQ} and the isomorphism~\eqref{eq:isomreg}.
Thus $\IC_{vt^{\xi+1}, W_\xi(w)}$ is in fact $A(w)$-equivariant and so is $\cL'_{W_\xi(w)}$.
\end{proof}

In particular, the functor $\For_{G(w)}^{A(w)}$ gives a $\kk$-algebra isomorphism:
\begin{equation} \label{eq:GA}
\Hom^\bullet_{G(w)}(\cL'_{W_\xi(w)}, \cL'_{W_\xi(w)})^\wedge \simeq 
\Hom^\bullet_{A(w)}(\cL'_{W_\xi(w)}, \cL'_{W_\xi(w)})^\wedge.
\end{equation} 

\subsection{Fourier--Laumon transform}
\label{Ssec:FL}
We regard the vector spaces $X'(w)$ as an $(A(w) \times \C^\times)$-variety, where $\C^\times$ acts simply by the scalar multiplication.
Note that this action factors through the surjective homomorphism 
\begin{equation} \label{eq:ACtoA}
A(w) \times \C^\times \to A(w) \quad \text{given by $(g,s) \mapsto g \cdot (\id_{P^{w(1)}}, s \id_{I^{w(0)}})$.}
\end{equation} 
The space $X(w) = X'(w)^*$ is also viewed an $(A(w) \times \C^\times)$-variety via the dual action.
Consider the $A(w)$-equivariant Fourier--Laumon transform 
\[ \Phi_{X'(w)} \colon D^b_{A(w) \times \C^\times}(X'(w), \kk) \xrightarrow{\simeq} D^b_{A(w) \times \C^\times}(X(w), \kk)\]
introduced in \cite{Lau} (see also \cite[\S6.9]{Ach}).
We define 
\[\cL_w \seq \Phi_w(\cL'_{W_\xi(w)}), \quad \text{where $\Phi_w \seq \For_{A(w)}^{A(w) \times \C^\times} \circ \Phi_{X'(w)} \circ \Infl_{A(w)}^{A(w) \times \C^\times}$}.\]
Since $\Phi_w$ sends $A(w)$-equivariant simple perverse sheaves on $X'(w)$ bijectively to the ones on $X(w)$ {(cf.\ \cite[Theorems 6.9.9 and 6.9.12]{Ach})}, $\cL_w$ is an $A(w)$-equivariant semisimple perverse sheaf on $X(w)$.
Letting $\IC_{v,w} \seq \Phi_w(\IC_{vt^{\xi+1}, W_\xi(w)})$ and $L_{v,w} \seq L_{vt^{\xi+1}, W_\xi(w)}$, 
the functor $\Phi_w$ translates \eqref{eq:DT} into
\[\cL_w \simeq \bigoplus_v L_{v,w} \boxtimes \IC_{v,w}
\]
where $v$ runs over the set of elements $v \in \N^\II$ satisfying $vt^{\xi+1} \in \Dom (W_\xi(w))$. 
Since $X(w)$ has finitely many $A(w)$-orbits and the stabilizer $\Aut(\phi)$ of each closed point $\phi$ is connected,  the intersection cohomology complexes of orbit closures exhaust the simple $A(w)$-equivariant perverse sheaves on $X(w)$. 
When $v=0$, we have 
\[\IC_{0,w} = \Phi_w(\uk_{\{0\}}) \simeq \uk_{X(w)}[\dim X(w)].\]

We define a $\kk$-algebra homomorphism $\varphi_w \colon U_q(\Lfg) \to \Hom^\bullet_{A(w)}(\cL_w, \cL_w)^\wedge$ to be the following composition:
\begin{align*}
\varphi_w \colon U_q(\Lfg) & \xrightarrow{\varphi'_{W_\xi(w)}} \Hom^\bullet_{G(w)}(\cL'_{W_\xi(w)}, \cL'_{W_\xi(w)})^\wedge \\ &\xrightarrow{\eqref{eq:GA}} \Hom^\bullet_{A(w)}(\cL'_{W_\xi(w)}, \cL'_{W_\xi(w)})^\wedge \xrightarrow{\Phi_w} \Hom^\bullet_{A(w)}(\cL_w, \cL_w)^\wedge. \end{align*}

\subsection{Deformed simple modules}
\label{Ssec:LoverX}
For the sake of brevity, we set 
\[{T(w) \seq T(W_\xi(w)) \subset G(w) = G(W_\xi(w)) \quad \text{and} \quad M_\xi(w) \seq M(\varpi_{W_\xi(w)}).}\] 
The latter is compatible with the notation~\eqref{eq:Lxi} as we have $L_\xi(w) = L(\varpi_{W_\xi(w)})$ (compare \eqref{eq:Lxi} with $\eqref{eq:varpiW}$).
Recall the generic element $\phi_\xi(w) \in X(w)$ from Definition~\ref{Def:phiw}.
For a closed point $\phi \in X(w)$, let $i_{\phi} \colon \{ \phi\} \to X(w)$ denote the inclusion. 

\begin{Prop} \label{Prop:LoverX}
We have
\[ 
\varphi_w^*\hrH^\bullet_{T(w)}(\cL_w) \simeq M_\xi(w)[\![ u ]\!], 
\quad \varphi_w^*\hrH^\bullet_{T(w)}(i_{\phi_\xi(w)}^*\cL_w) \simeq L_\xi(w)[\![u]\!]
\]
as $U_q(\Lfg)[\![u]\!]$-modules, where the action of $u$ on the left-hand sides is given by  the product with a non-zero element of $\rH^2_{T(w)}(\pt, \kk)$.\end{Prop}
\begin{proof} 
In this proof, we abbreviate $T(w)$, $X(w)$, $i_{\phi_\xi(w)}$ as $T$, $X$, $i$ respectively.
The first isomorphism follows from Theorem~\ref{Thm:standard} through the transform $\Phi_w$. In fact, as $\Phi_w(\uk_{\{0\}}) \simeq \uk_X[\dim X]$, we have
\[ \hrH^\bullet_{T}(i_0^!\cL'_{W_\xi(w)}) = \Hom^\bullet_T(\uk_{\{0\}}, \cL'_{W_\xi(w)})^\wedge \overset{\Phi_w}{\simeq} 
\Hom^\bullet_T(\uk_X, \cL_{w})^\wedge \simeq \hrH^\bullet_T(\cL_w).\]

We shall show the second isomorphism.
The functor $i^*$ yields a homomorphism
\[\hrH^\bullet_{T}(\cL_w) = \Hom_{T}^\bullet(\uk_{X}, \cL_w)^\wedge \xrightarrow{a} \Hom_{T}^\bullet(i^*\uk_{X}, i^*\cL_w)^\wedge = \hrH_{T}^\bullet(i^*\cL_w)\]
of modules over the Yoneda algebra $\Hom_{A(w)}^\bullet(\cL_w, \cL_w)^\wedge$.
We recall that
$A(w) \phi_\xi(w)$ is the unique open $A(w)$-orbit in $X(w)$ and the stabilizer of $\phi_\xi(w)$ is connected.
Therefore, the constant perverse sheaf $\uk_{X}[\dim X]$ is the unique simple object of $\Perv_{A(w)}(X(w))$ whose stalk at $\phi_\xi(w)$ is non-zero.
Thus, we have  $i^*\IC_{v,w} = 0$ if $v \neq 0$ and hence $i^* \cL_w 
= L_{0,w} \boxtimes i^*\uk_{X}[\dim X]$.
Now, we see that
$\hrH_{T}^\bullet(i^*\cL_w) \simeq L_{0,w} \otimes \kk[\![u]\!]$ as $\kk[\![u]\!]$-modules, and that $a$ is surjective as its restriction to the summand $L_{0,w} \boxtimes \IC_{0, w} \subset \cL_w$ yields an isomorphism.
As a $U_q(\Lfg)$-module, $\hrH_{T}^\bullet(i^*\cL_w)$ is a limit of iterated self-extensions of the simple module $L_\xi(w)$ {since $\varphi_w^*L_{0,w} \simeq L_\xi(w)$ by the last assertion of Theorem~\ref{Thm:Nak}}.

Let $N_\xi(w)$ denote the kernel of the quotient homomorphism $M_\xi(w) \to L_\xi(w)$.
The $U_q(\Lfg)$-module $N_\xi(w)$ does not contain $L_\xi(w)$ as its composition factor {because $N_\xi(w)$ does not have the $\ell$-highest weight vector of $M_\xi(w)$}. 
Since $M \mapsto M [\![u]\!]$ is an exact functor, we have a short exact sequence
\[ 0 \to N_\xi(w)[\![u]\!] \xrightarrow{b} M_\xi(w)[\![u]\!] \xrightarrow{c} L_\xi(w)[\![u]\!] \to 0\]
of $U_q(\Lfg)[\![u]\!]$-modules. 

We compare the homomorphisms $a$ and $c$.
For any positive integer $n$, we consider the base change from $\kk[\![u]\!]$ to the truncated polynomial ring $\kk [\![u]\!]/(u^n)$ to obtain the solid arrows in the following diagram:
\[
\xymatrix{
N_\xi(w)[\![u]\!]/(u^n) \  \ar@{^{(}->}[r]^-{b_n}& M_\xi(w)[\![u]\!]/(u^n) \ar@{->>}[r]^-{c_n} \ar[d]^-{\simeq}& L_\xi(w)[\![u]\!]/(u^n) \ar@{-->}[d]^-{\theta_n} \\
& \hrH^\bullet_{T}(\cL_w) / u^n \hrH^\bullet_{T}(\cL_w) \ar@{->>}[r]^-{a_n}& \hrH^\bullet_{T}(i^*\cL_w) / u^n \hrH^\bullet_{T}(i^*\cL_w), 
}
\]
where the upper row is exact.
We know that both $L_\xi(w)[\![u]\!]/(u^n) $ and $\hrH^\bullet_{T}(i^*\cL_w) / u^n \hrH^\bullet_{T}(i^*\cL_w)  {\simeq L_{0,w}\otimes \kk[\![u]\!]/(u^n)}$ are iterated self-extensions of $L_\xi(w)$ of the same length $n$ (as $U_q(\Lfg)$-modules) and that the image of $b_n$ does not contain $L_\xi(w)$ as a composition factor. 
Therefore, there exists a unique isomorphism $\theta_n$ of $U_q(\Lfg)$-modules represented by the dashed arrow in the diagram. 
Taking the limit $n \to \infty$, we get the desired isomorphism of $U_q(\Lfg)[\![u]\!]$-modules.
\end{proof}

\subsection{Tensor product}
\label{Ssec:LLoverX}
Let $w, w' \in \NII$. 
We fix decompositions $I^{(w+w')(k)} = I^{w(k)} \oplus I^{w'(k)}$, $k \in \{ 0,1\}$, to get
\begin{equation} \label{eq:dec_X} 
X(w+w') = X(w) \oplus X(w') \oplus X(w,w')^+ \oplus X(w,w')^-,
\end{equation}
where $X(w,w')^+ \seq \Hom_{Q_\xi}(I^{w(0)}, I^{w'(1)})$ and $X(w,w')^- \seq \Hom_{Q_\xi}(I^{w'(0)}, I^{w(1)})$.
Then we have $X(w+w')^{T(w')} = X(w) \oplus X(w')$.
Consider the diagram
\begin{equation} \label{eq:diag_X}
X(w) \oplus X(w') \xleftarrow{p_\pm} X(w)\oplus X(w')\oplus X(w,w')^{\pm} \xrightarrow{h_\pm} X(w+w'),
\end{equation}
where $h_\pm$ is the inclusion and $p_\pm$ is the projection.

\begin{Prop} \label{Prop:LLoverX}
Let $\phi \seq \phi_\xi(w) \oplus \phi_\xi(w') \in X(w) \oplus X(w')$. 
We have 
\begin{align*} 
\varphi_{w+w'}^* \hrH^\bullet_{T(w')}(i_{\phi}^*(p_+)_! (h_+)^*\cL_{w+w'}) & \simeq L_\xi(w) \otimes (L_\xi(w')[\![u]\!]), \\
\varphi_{w+w'}^* \hrH^\bullet_{T(w')}(i_{\phi}^*(p_-)_! (h_-)^* \cL_{w+w'}) & \simeq (L_\xi(w')[\![u]\!]) \otimes L_\xi(w).
\end{align*}
as $U_q(\Lfg)[\![u]\!]$-modules, where the action of $u$ on the left-hand side is given by  the product with a non-zero element of $\rH^2_{T(w')}(\pt, \kk)$.\end{Prop}
\begin{proof}
The assertion follows from Theorem~\ref{Thm:NakT}, Proposition~\ref{Prop:LoverX} and an analog of \cite[Proposition~10.1.2]{LusB}. 
For completeness, we give some details. 
Consider the decomposition 
\[ X'(w+w') = X'(w) \oplus X'(w') \oplus X'(w,w')^+ \oplus X'(w,w')^-,\]
where $X'(w,w')^+ \seq \Hom_{Q_\xi}(P^{w(1)}, I^{w'(0)})$ and $X'(w,w')^- \seq \Hom_{Q_\xi}(P^{w'(1)}, I^{w(0)})$.
Then we have $X'(w+w')^{T(w')} = X'(w) \oplus X'(w')$.
Under the identification $X'(w+w') = \qv_0(W_\xi(w+w'))$ in Proposition~\ref{Prop:NKQ}, we have
\[ \fA^\pm(W_\xi(w), W_\xi(w')) = X'(w) \oplus X'(w') \oplus X'(w,w')^\pm,\]
and $p'_\pm$ in \eqref{eq:diag_res} is the projection to $X'(w) \oplus X'(w')$.
Through the Nakayama functor, $X(w,w')^\pm$ is dual to $X'(w,w')^\mp$.
The dual of the diagram~\eqref{eq:diag_res} is identified with 
\[ X(w) \oplus X(w') \xrightarrow{{}^t p'_\pm} X(w) \oplus X(w') \oplus X(w,w')^\mp \xleftarrow{{}^t h'_\pm} X(w+w') \]
where ${}^t p'_\pm$ is the inclusion and ${}^t h'_\pm$ is the projection with respect to the decomposition~\eqref{eq:dec_X}.
By \cite[Proposition 6.9.13]{Ach}, we have
\begin{equation}  \label{eq:ph1}
(\Phi_w \boxtimes \Phi_{w'}) \circ (p'_\pm)_! \circ (h'_\pm)^* \simeq ({}^t p'_\pm)^* \circ ({}^t h'_\pm)_! \circ \Phi_{w+w'} [-d_\pm],
\end{equation}
where $d_\pm \seq \dim X(w,w')^{\pm} - \dim X(w,w')^\mp$. 
Since the diagram
\[ \xymatrix{
X(w) \oplus X(w') \oplus X(w,w')^\pm \ar[r]^-{h_\pm} \ar[d]_-{p_\pm} & X(w+w') \ar[d]^-{{}^t h'_\pm} \\
X(w) \oplus X(w') \ar[r]^-{{}^t p'_\pm} & X(w) \oplus X(w') \oplus X(w,w')^\mp
}\]
is cartesian, we have the base change isomorphism 
\begin{equation} \label{eq:ph2}
({}^t p'_\pm)^* \circ ({}^t h'_\pm)_! \simeq (p_\pm)_! \circ (h_\pm)^*.
\end{equation}
Combining \eqref{eq:ph1} with \eqref{eq:ph2}, we see that the Fourier--Laumon transform induces an isomorphism
\[
\hrH^\bullet_{T(w')}(i_0^!(p'_\pm)_! (h'_\pm)^*\cL'_{W_\xi(w+w')}) \simeq 
\hrH^\bullet_{T(w')}((p_\pm)_! (h_\pm)^*\cL_{w+w'}).
\]
By Theorem~\ref{Thm:NakT} and the first isomorphism in Proposition~\ref{Prop:LoverX} (or rather, by the discussion before Theorem~\ref{Thm:NakT}), we get an isomorphism
\[ \varphi_{w+w'}^* \hrH^\bullet_{T(w')}((p_+)_! (h_+)^*\cL_{w+w'}) \simeq 
\varphi_w^* \rH^\bullet(\cL_w) \otimes \varphi_{w'}^*\hrH^\bullet_{T(w')}(\cL_{w'}) 
 \]
of $U_q(\Lfg)[\![u]\!]$-modules.
Applying the functor $i_\phi^* \simeq i_{\phi_\xi(w)}^* \boxtimes i_{\phi_\xi(w')}^*$, 
we obtain 
\[ \varphi_{w+w'}^* \hrH^\bullet_{T(w')}(i_\phi^*(p_+)_! (h_+)^*\cL_{w+w'}) \simeq 
\varphi_w^* \rH^\bullet(i_{\phi_\xi(w)}^*\cL_w) \otimes \varphi_{w'}^*\hrH^\bullet_{T(w')}(i_{\phi_\xi(w')}^*\cL_{w'}). 
 \]
Together with the second isomorphism in Proposition~\ref{Prop:LoverX}, we get the first desired isomorphism. 
The other isomorphism is verified similarly.
\end{proof}

\subsection{Slice and E-invariant}
\label{Ssec:slice}
{Now we explain that the $E$-invariant comes into play as a transversal slice in the space $X(w)$.}

{Note that $\End_{Q_\xi}(I^{w(0)}) \oplus \End_{Q_\xi}(I^{w(1)})$ is the Lie algebra of $A(w)$  for any $w \in \NII$.} 
For any closed point $\phi \in X(w)$, we have an $A(w)$-equivariant linear map 
\[ f_\phi \colon \End_{Q_\xi}(I^{w(0)}) \oplus \End_{Q_\xi}(I^{w(1)}) \to X(w) \quad \text{given by $f_\phi(a,b) \seq b \circ \phi - \phi \circ a$.} \]
This is equal to the differential of the action map $A(w) \ni g \mapsto g \cdot \phi \in X(w)$ at $g = 1$ {(cf.\ \cite[p.215]{DF})}.
By the definition of the $E$-invariant, we have 
\[ X(w)/\Image f_\phi = E(\phi, \phi). \]
In particular, $\phi$ is rigid (i.e., $E(\phi, \phi) = 0$) if and only if the $A(w)$-orbit of $\phi$ is open in $X(w)$, that is when $\phi \simeq \phi_\xi(w)$.

In what follows, we consider the special case when 
$\phi = \phi_\xi(w) \oplus \phi_\xi(w')$ as in the previous section.
The decomposition~\eqref{eq:dec_X} induces the corresponding decomposition of $E(\phi, \phi)$. Namely,
letting 
\[\ep \colon X(w+w') \to X(w+w')/\Image f_\phi = E(\phi, \phi)\]
be the quotient map, we have
\begin{align*}
&\ep(X(w)) {=} E(\phi_\xi(w), \phi_\xi(w)) = 0, &&  \ep(X(w')) {=} E(\phi_\xi(w'), \phi_\xi(w')) = 0, \\
&\ep(X(w, w')^+) {=} E(\phi_\xi(w), \phi_\xi(w')), &&  \ep(X(w, w')^-) {=} E(\phi_\xi(w'), \phi_\xi(w)). 
\end{align*}
Choose a linear subspace $E^{\pm}$ of $X(w,w')^{\pm}$ stable under the action of the torus $T(w) \times T(w')$ such that the map $\ep$ restricts to isomorphisms $E^+ \simeq E(\phi_\xi(w), \phi_\xi(w'))$ and $E^- \simeq E(\phi_\xi(w'), \phi_\xi(w))$ respectively.
We define
\[ S \seq \phi + (E^+ \oplus E^-), \quad S^\pm \seq \phi + E^\pm, \]
which are affine subspaces of $X(w+w')$ stable under the action of the torus  $T(w) \times T(w')$. Let
\[ \{ \phi\} \xrightarrow{i_\pm} S^\pm \xrightarrow{j_\pm} S \xrightarrow{i_S} X(w+w')\]
denote the inclusions. 

\begin{Lem} \label{Lem:Sres}
With the above notation, we have a natural isomorphism 
\[ i_\phi^* (p _\pm)_! h_\pm^* \simeq i_{\pm}^! j_{\pm}^* i_S^* [{-}e_\pm] \]
of functors from $D^b_{A(w+w')}(X(w+w'),\kk)$ to $D^b_{T(w) \times T(w')}(\{\phi\}, \kk)$, where 
\[e_\pm \seq 2(\dim X(w,w)^\pm - \dim E^\pm).\]
\end{Lem}
\begin{proof}
A proof can be parallel to that of \cite[Lemma~7.7]{FH}.
{For the sake of completeness, we give the details.
In this proof, we put $T \seq T(w) \times T(w')$ and $X^\pm \seq X(w,w')^\pm$
for simplicity.
We have the commutative diagram:
\begin{equation} \label{eq:diagS}
\vcenter{
\xymatrix{
X(w)\oplus X(w')  
& \ar@{->}[l]_-{p_\pm} X(w)\oplus X(w')\oplus X^\pm \ar@{->}[r]^-{h_\pm}
& X(w+w') 
\\
\{\phi\} \ar@{->}[u]_-{i_\phi} \ar@<-3pt>@{->}[r]_-{i_2}
& \ar@<-3pt>@{->}[l]_-{p} \phi+X^\pm \ar@{->}[u]_-{i_1} 
& 
\\
\{\phi\} \ar@{=}[u] \ar@{->}[r]^-{i_\pm}
&  S^\pm=\phi+E^\pm\ar@{->}[r]^-{j_\pm} \ar@{->}[u]_-{i_3}
& S. \ar@{->}[uu]_{i_S} 
}}
\end{equation}
Here the arrow $p$ is the obvious map, and the arrows $i_1, i_2, i_3$ are the inclusions.
Note that the upper left square and the right square are both cartesian.
All the varieties in the diagram~\eqref{eq:diagS} are stable under the action of the $2$-dimensional torus $T$ and all the morphisms in the diagram~\eqref{eq:diagS} are $T$-equivariant.
By the base change and \cite[Proposition 2.3]{FW}, we have
\begin{equation} \label{eq:basechange}
i_\phi^*(p_\pm)_!h_\pm^*\simeq p_!i_1^*h_\pm^*\simeq i_2^!i_1^*h_\pm^*\simeq i_\pm^!i_3^!i_1^*h_\pm^*.
\end{equation}
Let $U^\pm$ be the unipotent subgroup of $A(w+w')$ whose Lie algebra is $H^\pm$, where 
\[ H^+ \seq \bigoplus_{k \in \{0,1\}}\Hom_{Q_\xi}(I^{w(k)},I^{w'(k)}), \quad 
H^- \seq \bigoplus_{k\in\{0,1\}}\Hom_{Q_\xi}(I^{w'(k)}, I^{w(k)}).\]
Both the varieties $X(w)\oplus X(w')\oplus X^\pm $ and $\phi+X^\pm$ are stable under the action of $U^\pm$, and hence they are $(U^\pm \rtimes T)$-varieties.
In particular, for any $\mathcal{F} \in D^b_{A(w+w')}(X(w+w'), \kk)$, the $*$-restriction $i_1^*h_\pm^* \mathcal{F}$ can be seen as an object of $D^b_{U^\pm \rtimes T}(\phi+X^\pm, \kk)$.
We shall show a natural isomorphism
\begin{equation} \label{eq:i3}
i_3^*  \simeq i_3^!  [e_\pm] 
\end{equation}
as functors from $D^b_{U^\pm \rtimes T}(\phi + X^\pm, \kk)$ to $D^b_{T}(S^\pm, \kk)$.
Consider the factorization $i_3 = \pi_3 \circ s_3$:
\[
\xymatrix{
 S^\pm \ar[rd]_-{s_3} \ar[rr]^-{i_3} & & \phi + X^\pm, \\
 & U^\pm \times S^\pm \ar[ur]_-{\pi_3} &
}
\]
where $s_3$ and $\pi_3$ are $T$-equivariant morphisms defined by $s_3 (x) \seq (1, x)$ and $\pi_3(g, x) \seq g\cdot x$.
The morphism $\pi_3$ is a locally trivial fibration. 
Indeed, its differential at the point $(1, \phi)$ is naturally identified with the linear map 
\[ H^\pm \oplus E^\pm \to X^\pm \quad \text{given by $(a,b,\psi) \mapsto f_\phi(a,b) + \psi$} \]
in the above notation.
This is surjective by our choice of $E^\pm$. 
Since the action of $\C^\times$ given via the cocharacter $\C ^\times \ni s \mapsto s^{\pm1} \id_{W_\xi(w')} \in T(w')$ contracts the variety $U^\pm \times S^\pm$ (resp.~$\phi+X^\pm$) to the single point $(1, \phi)$ (resp.~$\phi$), it follows that the morphism $\pi_3$ is surjective and its differential is surjective at any points.
Thus, $\pi_3$ is a locally trivial fibration with smooth fibers, and hence we have 
\[ \pi_3^* \simeq \pi_3^![2(\dim X^\pm - \dim (U^\pm \times S^\pm))]\]
as functors from $D^b_{U^\pm \rtimes T}(\phi+X^\pm, \kk)$ to $D^b_{U^\pm \rtimes T}(U^\pm \times S^\pm, \kk)$.
On the other hand, we have the induction equivalence 
\[ s_3^* \simeq s_3^! [2 (\dim (U^\pm \times S^\pm) - \dim S^\pm)] \colon D^b_{U^\pm \rtimes T}(U^\pm \times S^\pm, \kk) \xrightarrow{\sim} D^b_{T}(S^\pm, \kk).\]
Combining the above isomorphisms with the natural isomorphisms $i_3^* \simeq s_3^* \pi_3^*$ and $i_3^! \simeq s_3^! \pi_3^!$, we arrive at the isomorphism \eqref{eq:i3}.
Now, combining \eqref{eq:basechange} with \eqref{eq:i3}, we obtain
\[ i_\phi^*(p_\pm)_!h_\pm^*\simeq i_\pm^!i_3^!i_1^*h_\pm^*\simeq  i_\pm^!i_3^*i_1^*h_\pm^* [-e_\pm] \simeq i_{\pm}^! j_{\pm}^* i_S^* [-e_\pm]\]
as desired.}
\end{proof}

For $\mathcal{F} \in D^b_{A(w+w')}(X(w+w'), \kk)$, we define
\[\mathcal{F}|_S \seq i_S^* \mathcal{F}[\dim S - \dim X(w+w')].\]

\begin{Prop} 
\label{Prop:LLoverS}
With the above notation, we have 
\begin{align*} 
\varphi_{w+w'}^* \hrH^\bullet_{T(w')}((i_+)^! (j_+)^*\cL_{w+w'}|_S) & \simeq L_\xi(w) \otimes (L_\xi(w')[\![u]\!]), \\
\varphi_{w+w'}^* \hrH^\bullet_{T(w')}((i_-)^! (j_-)^* \cL_{w+w'}|_S) & \simeq (L_\xi(w')[\![u]\!]) \otimes L_\xi(w).
\end{align*}
as $U_q(\Lfg)[\![u]\!]$-modules.
\end{Prop}
\begin{proof}
The assertion follows from Proposition~\ref{Prop:LLoverX} and Lemma~\ref{Lem:Sres}.
{Note that the shift $[-e_\pm]$ appearing in Lemma~\ref{Lem:Sres} is irrelevant here since we forget the grading by taking the direct products of cohomologies over all the degrees.}
\end{proof}

\begin{Cor}
\label{Cor:fact}
Let $w,w' \in \NII$.
If we have 
\[{E(\phi_\xi(w), \phi_\xi(w')) = E(\phi_\xi(w'), \phi_\xi(w)) = 0,}\] then
$L_\xi(w)$ and $L_\xi(w')$ strongly commute. 
In particular, any simple module $L_\xi(w)$ of the category $\Cc_{\xi, 1}$ is real.
\end{Cor}
\begin{proof}
Under the assumption, we have $S = S^\pm = \{ \phi_\xi(w+w')\}$, and therefore $(i_+)^! (j_+)^*\cL_{w+w'}|_S$ coincides with $i_{\phi_\xi(w+w')}^*\cL_{w+w'}$ up to a shift.
Then Propositions~\ref{Prop:LoverX} \& \ref{Prop:LLoverS} yield
\[ L_\xi(w) \otimes L_\xi(w') \simeq \varphi_{w+w'}^*\rH^\bullet(i^*_{\phi_\xi(w+w')}\cL_{w+w'}) \simeq L_\xi(w+w').
\qedhere\]
\end{proof}

Note that the above Corollary~\ref{Cor:fact} and its converse follow from Theorem~\ref{Thm:mcat} as well, although we do not rely on it in our proof. 

\begin{Lem} \label{Lem:transv}
Assume that we have
\begin{equation} \label{eq:AssumpE}
E(\phi_\xi(w), \phi_\xi(w')) = 0 \quad \text{or} \quad E(\phi_\xi(w'), \phi_\xi(w)) = 0. 
\end{equation}
Then $S$ meets $A(w+w')$-orbits in $X(w+w')$ transversally. 
Moreover, we have $S \cap A(w+w')\phi = \{ \phi\}$ with $\phi = \phi_\xi(w) \oplus \phi_\xi(w')$.
\end{Lem}
\begin{proof}
Our discussion here mimics that of \cite[2.2]{GG} for the Slodowy slice.
By the symmetry, we may assume $E(\phi_\xi(w), \phi_\xi(w')) = 0$. Then, we have $E^+= \{0\}$ and $S = S^-$.
In particular, the action of the torus $T(w)$ contracts the whole $S$ to the unique fixed point $\phi$. 

Note that the differential $\mathrm{d} \alpha$ of the action map $\alpha \colon A(w+w') \times S \to X(w+w')$ at $(1,\phi)$ is identical to the map
\[ \End_{Q_\xi}(I^{(w+w')(0)}) \oplus \End_{Q_\xi}(I^{(w+w')(1)}) \oplus E \to X(w+w')\]
given by $(a,b,\psi) \mapsto f_\phi(a,b)+\psi$.
Since $\Image f_\phi \oplus E = X(w+w')$, this is surjective.  
Using the contracting action of the torus $T(w)$, we can conclude that the differential $\mathrm{d} \alpha$ is surjective at $(1,x)$ for any $x \in S$. 
This implies the first assertion.
The last assertion follows from an argument analogous to the proof of \cite[Proposition~3.7.15]{CG}.
\end{proof}

The following proposition plays a key role in the proof of our main theorem in the next subsection.

\begin{Prop} \label{Prop:key}
Under the assumption~\eqref{eq:AssumpE}, $\cL_{w+w'}|_S$ is a semisimple perverse sheaf on $S$ containing both $\uk_{S}[\dim S]$ and $\uk_{\{ \phi \}}$ as summands, where $\phi \seq \phi_\xi(w) \oplus \phi_\xi(w')$ as above.
\end{Prop}
\begin{proof}
For simplicity, we put $\tw \seq w +w'$ in this proof.
By Lemma~\ref{Lem:transv}, $S$ is a transversal slice through $\phi$. 
By \cite[Theorem~5.4.1]{GM2}, $\IC_{v,\tw}|_S$ is a simple perverse sheaf for any possible $v$ and hence $\cL_{\tw}$ a semisimple perverse sheaf. 
It contains $\IC_{0,\tw}|_S = \uk_{S}[\dim S]$ as a summand.   
It remains to show that $\IC_{v, \tw}|_S = \uk_{\{\phi\}}$ for some $v$.
To this end, it is enough to show that the intersection cohomology complex $\IC(\bar{O})$ of the closure of the orbit $O \seq A(\tw)\phi$ coincides with $\IC_{v, \tw}$ for some $v$ because we know $O \cap S = \{ \phi \}$ by the last assertion of Lemma~\ref{Lem:transv}.
In view of Proposition~\ref{Prop:NKQ}, it suffices to show that a shift of $\IC(\bar{O})$ appears as a summand of $\Phi_{\tw}((p_{v, \tw})_!\uk_{F(v,\tw)})$ for a suitable $v \in \N^\II$. 

By symmetry, we may assume $E(\phi_\xi(w'), \phi_\xi(w)) = 0$.
Put $K \seq \Ker(\phi_\xi(w))$ and $K' \seq \Ker(\phi_\xi(w'))$. 
By \eqref{eq:Einv}, we know that $\Ext^1_{Q_\xi}(K,K)$, $\Ext^1_{Q_\xi}(K',K')$ and $\Ext^1_{Q_\xi}(K', K)$ all vanish. 
Let $v \in \N^\II$ be the dimension vector of $K$.
We shall show that a shift of $\IC(\bar{O})$ appears in $\Phi_{\tw}((p_{v, \tw})_!\uk_{F(v,\tw)})$ for this $v$.  
By definition, $F(v,\tw)$ is a vector subbundle of the trivial bundle $\Gr_{v}(I^{\tw(0)}) \times X'(\tw)$ over the quiver Grassmannian $\Gr_{v}(I^{\tw(0)})$. 
Let $F(v, \tw)^\perp$ denote its annihilator subbundle in $\Gr_{v}(I^{\tw(0)}) \times X(\tw)$. 
By \cite[Lemma~3.1.7]{KQ}, it is described as
\begin{equation} 
F(v, \tw)^\perp = \{ (N, \psi) \in \Gr_{v}(I^{\tw(0)}) \times X(\tw) \mid N \subset \Ker \psi \}.
\end{equation}
By \cite[Corollary~6.9.14 \& Proposition~6.9.15]{Ach}, $\Phi_{\tw}((p_{v, \tw})_!\uk_{F(v,\tw)})$ is isomorphic to a shift of $(p_{v, \tw}^\perp)_!\uk_{F(v',\tw)^\perp}$, where $p_{v, \tw}^\perp \colon F(v, \tw)^\perp \to X(\tw)$ denotes the second projection $(N, \psi) \mapsto \psi$. 

Now, we have to prove
that a shift of $\IC(\bar{O})$ occurs in $(p_{v, \tw}^\perp)_!\uk_{F(v,\tw)^\perp}$.  
We need additional notations. 
For $a, b \in \N^\II$ (resp.\ $\NII$), we write $a \le b$ if $a_i \le b_i$ for all $i \in \II$ (resp.\ $a_i(k) \le b_i(k)$ for all $(i,k) \in \II \times \{ 0,1\}$).
For $M \in \rep \C Q_\xi$, we define its Betti vector $b_{M} = (b_M(0), b_M(1)) \in \NII$ by $b_{M,i}(k) \seq \dim \Ext_{Q_\xi}^k(S_i, M)$ for $k \in \{ 0,1\}$ and $i \in \II$. 
Let $\rho_M \in X(b_M)$ denote the minimal injective resolution of $M$. 
For any $a = (a(0), a(1)) \in \NII$ such that $a(0) \le a(1)$, let $\nu_a \in X(a)$ be an injection $I^{a(0)} \to I^{a(1)}$. 
With these notations, it is easy to see that any $\psi \in X(w)$ decomposes as $\psi \simeq \rho_{M} \oplus \nu_{w - b_M}$ with $M = \Ker \psi$.

Let us consider the subset $U$ of $F(v,\tw)^\perp$ consisting of pairs $(N,\psi)$ such that (i) $N \simeq K$, (ii) $\Ext^1_{Q_\xi}(\Ker \psi/ N, K\oplus \Ker \psi/ N) = 0$, and (iii) $b_{\Ker \psi / N} \le b_{K'}$. 
Since $K$ is a generic representation and the functions mapping $(N,\psi) \in F(v,\tw)^\perp$ to $\dim\Ext^1_{Q_\xi}(\Ker \psi/ N, K \oplus \Ker \psi/ N)$ and $b_{\Ker \psi/N}$ are upper semi-continuous, $U$ is an open subset.
Moreover, it is non-empty as $(K,\phi) \in U$ and hence dense in the smooth connected variety $F(v,\tw)^\perp$. 
We claim that, for any $(N, \psi) \in U$, there is an isomorphism $\psi \simeq \phi$.
Once the claim is verified, we have $p_{v,\tw}^\perp(U) = O$, which implies $p_{v,\tw}^\perp(F(v,\tw)^\perp) = \bar{O}$. 
Therefore a shift of $\IC(\bar{O})$ must contribute to $(p_{v, \tw}^\perp)_!\uk_{F(v,\tw)^\perp}$ as desired.

We prove the claim. 
Assume $(N, \psi) \in U$.
By the conditions (i) and (ii), we have $\Ker \psi \simeq K \oplus C$, where $C \seq \Ker \psi / N$. 
Then, we have 
\[\psi \simeq \rho_{K \oplus C} \oplus \nu_{\tw-b_{K \oplus C}} = \rho_{K} \oplus \rho_{C} \oplus \nu_{w - b_K} \oplus \nu_{w' - b_C}.\]
As $\phi_\xi(w) \simeq \rho_K \oplus \nu_{w - b_K}$, we have $\psi \simeq \phi_\xi(w) \oplus \psi'$, where $\psi' \seq \rho_C \oplus \nu_{w'-b_C} \in X(w')$.
Since $\phi_\xi(w')$ is in the unique open orbit in $X(w')$, it follows that $b_C = b_{\Ker \psi'} \ge b_{\Ker \phi_\xi(w')} = b_{K'}$. 
The condition (iii) forces $b_C = b_{K'}$, which implies that $C$ shares the same dimension vector as $K'$. 
Again by (ii), we have $\Ext^1_{Q_\xi}(C,C) = 0$ and hence $C \simeq K'$.
This implies $\psi' \simeq \rho_{K'} \oplus \nu_{w' - b_{K'}} \simeq \phi_\xi(w')$.
Thus, we obtain $\psi \simeq \phi_{\xi}(w) \oplus \psi' \simeq \phi_\xi(w) \oplus\phi_\xi(w') = \phi$. 
\end{proof}

\subsection{Proof of Theorem~\ref{Thm:main}} \label{Ssec:proof}
Our goal is to show the equality 
\begin{equation} \label{eq:goal}
\oo(L_\xi(w), L_\xi(w')) = \dim E(\phi_\xi(w), \phi_\xi(w'))
\end{equation}
for any $w,w' \in \NII$.
We first prove it under the assumption~\eqref{eq:AssumpE}, where we obtain a sheaf theoretic interpretation of $R$-matrices as a byproduct.

\begin{Prop} \label{Prop:last}
Under the assumption~\eqref{eq:AssumpE}, the equality~\eqref{eq:goal} holds.
\end{Prop}
\begin{proof}
For simplicity, we put $L \seq L_\xi(w)$ and $L' \seq L_\xi(w')$ in this proof.
Let $\phi \seq \phi_\xi(w) \oplus \phi_\xi(w')$ as before and $i \colon \{ \phi \} \to S$ denote the inclusion.
We have the following morphisms arising from the adjunction unit/counit:
\[  \eta \colon \uk_{S} [d_S] \to i_* i^* \uk_S[d_S] = \uk_{\{\phi\}}[d_S], \quad
\vep \colon \uk_{\{\phi\}} = i_! i^! \uk_S [2d_S] \to \uk_S[2d_S],\] 
where $d_S \seq \dim S$.
We also abbreviate $\cL_{w+w'}|_S$ as $\cL$, and $T(w')$ as $T$.

First, we consider the case when $E(\phi_\xi(w), \phi_\xi(w')) = 0$. 
In this case, we have $S^+ = \{\phi\}$, $S^- = S$ and hence $(i_+)^! (j_+)^*\cL = i^*\cL$, $(i_-)^! (j_-)^*\cL = i^! \cL$. 
Moreover, we have $i^*\cL \simeq p_* \cL$ with $p \colon S \to \{\phi\}$ being the obvious morphism (cf.~\cite[Proposition~2.3]{FW}).
By Proposition~\ref{Prop:LLoverS}, we have
\begin{align}
L \otimes L'[\![u]\!] &\simeq \hrH^\bullet_{T}(p_*\cL) \simeq \Hom^\bullet_{T}(\uk_S[d_S], \cL)^\wedge, \label{eq:LL'}\\
L'[\![u]\!] \otimes L & \simeq \hrH^\bullet_{T}(i^!\cL) \simeq  \Hom^\bullet_{T}(\uk_{\{0\}},\cL)^\wedge. \label{eq:L'L}
\end{align}
Choose $\ell$-highest weight vectors $v_L \in L$ and $v_{L'} \in L'$.  
We shall identify the images of $v_L \otimes v_{L'}$ (resp.\ $v_{L'} \otimes v_L$) under the isomorphism~\eqref{eq:LL'} (resp.\ \eqref{eq:L'L}). 
Recall the isomorphism $\varphi_{w+w'}^* L_{0,w+w'} \simeq L_\xi(w+w')$ from Theorem~\ref{Thm:Nak}.
Consider the $1$-dimensional subspace $(L_{0,w+w'})_0 \subset L_{0,w+w'}$ corresponding to the $\ell$-highest weight space of $L_\xi(w+w')$. 
We have the embedding of the corresponding summand $\iota \colon \uk_S[d_S] = (L_{0,w+w'})_0 \boxtimes \uk_S[d_S] \subset \cL$.
By construction, this contributes to the $\ell$-highest weight spaces of $L\otimes L'[\![u]\!]$ and $L'[\![u]\!] \otimes L$. 
More precisely, we have the following commutative diagrams:
\[
\xymatrix{
L\otimes L'[\![u]\!] \ar[r]^-{\simeq}_-{\eqref{eq:LL'}} 
&\Hom^\bullet_T(\uk_S[d_S], \cL)^\wedge \\
\kk[\![u]\!] (v_L \otimes v_{L'}) \ar[r]^-{\simeq} \ar@{^{(}->}[u]^-{\text{inclusion}} 
& \Hom^\bullet_T(\uk_S[d_S], \uk_S[d_S])^\wedge, \ar@{^{(}->}[u]^-{\iota_*} \\
}
\]
\[
\xymatrix{
L'[\![u]\!]\otimes L \ar[r]^-{\simeq}_-{\eqref{eq:L'L}} 
&\Hom^\bullet_T(\uk_{\{\phi\}}, \cL)^\wedge \\
\kk[\![u]\!] (v_{L'} \otimes v_{L}) \ar[r]^-{\simeq} \ar@{^{(}->}[u]^-{\text{inclusion}} 
& \Hom^\bullet_T(\uk_{\{\phi\}}, \uk_S[d_S])^\wedge, \ar@{^{(}->}[u]^-{\iota_*} \\
}
\]
where $\iota_*$ means the post-composition with $\iota$.
Since $\Hom^\bullet_T(\uk_S[d_S], \uk_S[d_S])^\wedge$ is generated over $\kk[\![u]\!]$ by the identity $\id_{\uk_S[d_S]} \in \Hom_T^0(\uk_S[d_S], \uk_S[d_S])$, we may assume that the isomorphism~\eqref{eq:LL'} sends the vector $v_L \otimes v_{L'}$ to $\id_{\uk_S[d_S]}$.  
By the same reason, the isomorphism~\eqref{eq:L'L} sends the vector $v_{L'}\otimes v_L$ to $\vep \in \Hom_T^{d_S}(\uk_{\{\phi\}}, \uk_{S}[d_S])$.
Then, the following diagram commutes:
\[
\xymatrix{
L\otimes L'(\!(u)\!) \ar@{<-^{)}}[r] \ar[d]_-{\wh{R}_{L,L'}}&
L\otimes L'[\![u]\!] \ar[r]^-{\simeq}_-{\eqref{eq:LL'}} 
&\Hom^\bullet_T(\uk_S[d_S], \cL)^\wedge \ar[d]^-{\vep^*}\\
L'(\!(u)\!) \otimes L \ar@{<-^{)}}[r] &
L'[\![u]\!]\otimes L \ar[r]^-{\simeq}_-{\eqref{eq:L'L}} 
&\Hom^\bullet_T(\uk_{\{\phi\}}, \cL)^\wedge, 
}
\]
where $\vep^*$ denotes the pre-composition with $\vep$.
Indeed, the homomorphism $\vep^*$ intertwines the $U_q(\Lfg)[\![u]\!]$-actions (given through $\varphi_{w+w'}$), and sends $\iota_*\id_{\uk_S[d_S]}$  ($=$ the image of $v_L \otimes v_{L'}$) to $\iota_*\vep$ ($=$ the image of $v_{L'} \otimes v_{L}$). 
{As such, it should correspond to $\wh{R}_{L,L'}$ after the localization by the characterization of the normalized $R$-matrix (see \S\ref{Ssec:Rmat}).}
The above commutative diagram implies $\wh{R}_{L,L'}(L \otimes L'[\![u]\!]) \subset L'[\![u]\!]\otimes L$. 
By Remark~\ref{Rem:Mu}, we obtain $\oo(L,L') = 0 = E(\phi_\xi(w), \phi_\xi(w'))$ as desired.

Next, we consider the remaining case when $E(\phi_\xi(w'), \phi_\xi(w))=0$.
Then we have $S^+ = S$ and $S^- = \{\phi\}$.
Similarly, there are isomorphisms
\begin{align}
L \otimes L'[\![u]\!] & \simeq \Hom^\bullet_{T}(\uk_{\{\phi\}},\cL)^\wedge, \label{eq:LL'2}\\
L'[\![u]\!] \otimes L & \simeq  
\Hom^\bullet_{T}(\uk_S[d_S], \cL)^\wedge, \label{eq:L'L2}
\end{align}
under which the vector $v_L \otimes v_{L'}$ corresponds to $\iota_*\vep$, and the vector $v_{L'} \otimes v_L$ corresponds to $\iota_* \id_{\uk_S[d_S]}$.
Let $c \in \kk$ be a scalar determined by the equation 
\[\vep \circ \eta = c u^{d_S} \id_{\uk_S[d_S]}\] 
in $\Hom^{2d_S}_T(\uk_S[d_S], \uk_S[d_S]) = \rH_T^{2d_S}(S, \kk)$.
We have $c \neq 0$ as $c u^{d_S}$ corresponds to the $T$-equivariant Euler class of $E$.
Then, the diagram 
\[
\xymatrix{
L\otimes L'(\!(u)\!) \ar@{<-^{)}}[r] \ar[d]_-{cu^{d_S}\wh{R}_{L,L'}}&
L\otimes L'[\![u]\!] \ar[r]^-{\simeq}_-{\eqref{eq:LL'2}} 
&\Hom^\bullet_T(\uk_{\{\phi\}}, \cL)^\wedge \ar[d]^-{\eta^*}\\
L'(\!(u)\!) \otimes L \ar@{<-^{)}}[r] &
L'[\![u]\!]\otimes L \ar[r]^-{\simeq}_-{\eqref{eq:L'L2}} 
&\Hom^\bullet_T(\uk_S[d_S], \cL)^\wedge
}
\]
commutes because $\eta^*$ intertwines the $U_q(\Lfg)[\![u]\!]$-actions and sends $\iota_*\vep$ ($=$ the image of $v_{L} \otimes v_{L'}$) to
$\iota_*(cu^{d_S}\id_{\uk_S[d_S]})$  ($=$ the image of $cu^{d_S}v_L \otimes v_{L'}$).  
Moreover, the specialization of $\eta^*$ at $u=0$ is equal to its non-equivariant version $\eta^* \colon \Hom^\bullet(\uk_{\{\phi\}}, \cL) \to \Hom^\bullet(\uk_S[d_S], \cL)$ (use Lemma~\ref{Lem:eq_change}), which is non-zero as $\cL$ contains $\uk_{\{\phi\}}$ as a summand by Proposition~\ref{Prop:key}. 
Therefore, we have $u^{d_S} \wh{R}_{L,L'}(L \otimes L'[\![u]\!]) \subset L'[\![u]\!] \otimes L$ and $(u^{d_S} \wh{R}_{L,L'})|_{u= 0} \neq 0$.
By Remark~\ref{Rem:Mu}, we obtain  $\oo(L,L') = d_S = \dim E(\phi_\xi(w), \phi_\xi(w'))$.
\end{proof}

Finally, we treat the general case.
Let $\phi_\xi(w) = \phi_\xi(w^{(1)}) \oplus \cdots \oplus \phi_\xi(w^{(l)})$  and $\phi_\xi(w') = \phi_\xi(w'^{(1)}) \oplus \cdots \oplus \phi_\xi(w'^{(l')})$ be decompositions in $C^2(\inj \C Q_\xi)$ with all the summands indecomposable. 
Then, by Corollary~\ref{Cor:fact}, we have the corresponding factorizations $L_\xi(w) \simeq L_\xi(w^{(1)}) \otimes \cdots \otimes L_\xi(w^{(l)})$ and $L_\xi(w') \simeq L_\xi(w'^{(1)}) \otimes \cdots \otimes L_\xi(w'^{(l')})$.
By Lemma~\ref{Lem:omult}, we have
\begin{equation} \label{eq:sum1}
\oo(L_\xi(w), L_\xi(w')) = \sum_{1 \le k \le l} \sum_{1 \le k' \le l'}
\oo(L_\xi(w^{(k)}), L_\xi(w'^{(k')})). 
\end{equation}
On the other hand, we have an obvious isomorphism
\begin{equation} \label{eq:sum2}
E(\phi_\xi(w), \phi_\xi(w')) \simeq \bigoplus_{1 \le k \le l} \bigoplus_{1 \le k' \le l'} E(\phi_\xi(w^{(k)}), \phi_\xi(w'^{(k')})).  
\end{equation}
For indecomposables, the assumption~\eqref{eq:AssumpE} is always satisfied by Remark~\ref{Rem:order} and \eqref{eq:EE}. Therefore, by Proposition~\ref{Prop:last}, we have 
\[ \oo(L_\xi(w^{(k)}), L_\xi(w'^{(k')})) = \dim E(\phi_\xi(w^{(k)}), \phi_\xi(w'^{(k')}))\] 
for any $1 \le k \le l$ and $1 \le k' \le l'$.
Thus, together with \eqref{eq:sum1} and \eqref{eq:sum2}, we get \eqref{eq:goal} for general $w,w' \in \NII$, completing the proof of Theorem~\ref{Thm:main}.

\subsection*{Acknowledgments}
{The author would like to thank the anonymous referee for many helpful suggestions.}
This work was supported by JSPS KAKENHI Grant No.\ JP23K12955.   


\begin{thebibliography}{10}

\bibitem{Ach}
Pramod~N. Achar.
\newblock {\em Perverse sheaves and applications to representation theory},
  volume 258 of {\em Mathematical Surveys and Monographs}.
\newblock American Mathematical Society, Providence, RI, [2021] \copyright
  2021.

\bibitem{ASS}
Ibrahim Assem, Daniel Simson, and Andrzej Skowro\'nski.
\newblock {\em Elements of the representation theory of associative algebras.
  {V}ol. 1}, volume~65 of {\em London Mathematical Society Student Texts}.
\newblock Cambridge University Press, Cambridge, 2006.
\newblock Techniques of representation theory.

\bibitem{BFL}
Karin Baur, Changjian Fu, and Jian-Rong Li.
\newblock A correspondence between additive and monoidal categorifications with
  application to {G}rassmannian cluster categories.
\newblock Preprint, \arxiv{2410.04401}v2, 2024.

\bibitem{BL}
Joseph Bernstein and Valery Lunts.
\newblock {\em Equivariant sheaves and functors}, volume 1578 of {\em Lecture
  Notes in Mathematics}.
\newblock Springer-Verlag, Berlin, 1994.

\bibitem{Bra}
Tom Braden.
\newblock Hyperbolic localization of intersection cohomology.
\newblock {\em Transform. Groups}, 8(3):209--216, 2003.

\bibitem{BC}
Matheus Brito and Vyjayanthi Chari.
\newblock Tensor products and {$q$}-characters of {HL}-modules and monoidal
  categorifications.
\newblock {\em J. \'Ec. polytech. Math.}, 6:581--619, 2019.

\bibitem{CC}
Philippe Caldero and Fr\'ed\'eric Chapoton.
\newblock Cluster algebras as {H}all algebras of quiver representations.
\newblock {\em Comment. Math. Helv.}, 81(3):595--616, 2006.

\bibitem{Cao}
Peigen Cao.
\newblock F-invariant in cluster algebras.
\newblock Preprint, \arxiv{2306.11438}v4, 2025.

\bibitem{CW}
Sabin Cautis and Harold Williams.
\newblock Cluster theory of the coherent {S}atake category.
\newblock {\em J. Amer. Math. Soc.}, 32(3):709--778, 2019.

\bibitem{Chari}
Vyjayanthi Chari.
\newblock Braid group actions and tensor products.
\newblock {\em Int. Math. Res. Not.}, (7):357--382, 2002.

\bibitem{CP}
Vyjayanthi Chari and Andrew Pressley.
\newblock {\em A guide to quantum groups}.
\newblock Cambridge University Press, Cambridge, 1994.

\bibitem{CPweyl}
Vyjayanthi Chari and Andrew Pressley.
\newblock Weyl modules for classical and quantum affine algebras.
\newblock {\em Represent. Theory}, 5:191--223, 2001.

\bibitem{CG}
Neil Chriss and Victor Ginzburg.
\newblock {\em Representation theory and complex geometry}.
\newblock Birkh\"{a}user Boston, Inc., Boston, MA, 1997.

\bibitem{Contu}
Alessandro Contu.
\newblock Solution of a problem in monoidal categorification by additive
  categorification.
\newblock {\em J. Algebra}, 691:128--185, 2026.

\bibitem{DF}
Harm Derksen and Jiarui Fei.
\newblock General presentations of algebras.
\newblock {\em Adv. Math.}, 278:210--237, 2015.

\bibitem{DWZ}
Harm Derksen, Jerzy Weyman, and Andrei Zelevinsky.
\newblock Quivers with potentials and their representations {II}: applications
  to cluster algebras.
\newblock {\em J. Amer. Math. Soc.}, 23(3):749--790, 2010.

\bibitem{FW}
Peter Fiebig and Geordie Williamson.
\newblock Parity sheaves, moment graphs and the {$p$}-smooth locus of
  {S}chubert varieties.
\newblock {\em Ann. Inst. Fourier (Grenoble)}, 64(2):489--536, 2014.

\bibitem{FZ1}
Sergey Fomin and Andrei Zelevinsky.
\newblock Cluster algebras. {I}. {F}oundations.
\newblock {\em J. Amer. Math. Soc.}, 15(2):497--529, 2002.

\bibitem{FZ2}
Sergey Fomin and Andrei Zelevinsky.
\newblock Cluster algebras. {II}. {F}inite type classification.
\newblock {\em Invent. Math.}, 154(1):63--121, 2003.

\bibitem{FZy}
Sergey Fomin and Andrei Zelevinsky.
\newblock {$Y$}-systems and generalized associahedra.
\newblock {\em Ann. of Math. (2)}, 158(3):977--1018, 2003.

\bibitem{FZ4}
Sergey Fomin and Andrei Zelevinsky.
\newblock Cluster algebras. {IV}. {C}oefficients.
\newblock {\em Compos. Math.}, 143(1):112--164, 2007.

\bibitem{FR}
Edward Frenkel and Nicolai Reshetikhin.
\newblock The {$q$}-characters of representations of quantum affine algebras
  and deformations of {$\mathscr W$}-algebras.
\newblock In {\em Recent developments in quantum affine algebras and related
  topics ({R}aleigh, {NC}, 1998)}, volume 248 of {\em Contemp. Math.}, pages
  163--205. Amer. Math. Soc., Providence, RI, 1999.

\bibitem{Fahw}
Ryo Fujita.
\newblock Affine highest weight categories and quantum affine {S}chur-{W}eyl
  duality of {D}ynkin quiver types.
\newblock {\em Represent. Theory}, 26:211--263, 2022.

\bibitem{Frmat}
Ryo Fujita.
\newblock Graded quiver varieties and singularities of normalized
  {$R$}-matrices for fundamental modules.
\newblock {\em Selecta Math. (N.S.)}, 28(1):Paper No. 2, 45, 2022.

\bibitem{FH}
Ryo Fujita and David Hernandez.
\newblock Monoidal {J}antzen filtrations.
\newblock {\em Adv. Math.}, 495:Paper No. 110963, 81, 2026.

\bibitem{FM}
Ryo Fujita and Kota Murakami.
\newblock Deformed {C}artan matrices and generalized preprojective algebras
  {I}: {F}inite type.
\newblock {\em Int. Math. Res. Not. IMRN}, (8):6924--6975, 2023.

\bibitem{GG}
Wee~Liang Gan and Victor Ginzburg.
\newblock Quantization of {S}lodowy slices.
\newblock {\em Int. Math. Res. Not.}, (5):243--255, 2002.

\bibitem{GM2}
Mark Goresky and Robert MacPherson.
\newblock Intersection homology. {II}.
\newblock {\em Invent. Math.}, 72(1):77--129, 1983.

\bibitem{HL}
David Hernandez and Bernard Leclerc.
\newblock Cluster algebras and quantum affine algebras.
\newblock {\em Duke Math. J.}, 154(2):265--341, 2010.

\bibitem{HL2}
David Hernandez and Bernard Leclerc.
\newblock Monoidal categorifications of cluster algebras of type {$A$} and
  {$D$}.
\newblock In {\em Symmetries, integrable systems and representations},
  volume~40 of {\em Springer Proc. Math. Stat.}, pages 175--193. Springer,
  Heidelberg, 2013.

\bibitem{KKK}
Seok-Jin Kang, Masaki Kashiwara, and Myungho Kim.
\newblock Symmetric quiver {H}ecke algebras and {R}-matrices of quantum affine
  algebras.
\newblock {\em Invent. Math.}, 211(2):591--685, 2018.

\bibitem{KKKO}
Seok-Jin Kang, Masaki Kashiwara, Myungho Kim, and Se-jin Oh.
\newblock Simplicity of heads and socles of tensor products.
\newblock {\em Compos. Math.}, 151(2):377--396, 2015.

\bibitem{KKKOm}
Seok-Jin Kang, Masaki Kashiwara, Myungho Kim, and Se-jin Oh.
\newblock Monoidal categorification of cluster algebras.
\newblock {\em J. Amer. Math. Soc.}, 31(2):349--426, 2018.

\bibitem{KKOP}
Masaki Kashiwara, Myungho Kim, Se-jin Oh, and Euiyong Park.
\newblock Monoidal categorification and quantum affine algebras.
\newblock {\em Compos. Math.}, 156(5):1039--1077, 2020.

\bibitem{KKOPd}
Masaki Kashiwara, Myungho Kim, Se-jin Oh, and Euiyong Park.
\newblock Simply laced root systems arising from quantum affine algebras.
\newblock {\em Compos. Math.}, 158(1):168--210, 2022.

\bibitem{KQ}
Yoshiyuki Kimura and Fan Qin.
\newblock Graded quiver varieties, quantum cluster algebras and dual canonical
  basis.
\newblock {\em Adv. Math.}, 262:261--312, 2014.

\bibitem{Lau}
G\'erard Laumon.
\newblock Transformation de {F}ourier homog\`ene.
\newblock {\em Bull. Soc. Math. France}, 131(4):527--551, 2003.

\bibitem{LusB}
George Lusztig.
\newblock {\em Introduction to quantum groups}, volume 110 of {\em Progress in
  Mathematics}.
\newblock Birkh\"{a}user Boston, Inc., Boston, MA, 1993.

\bibitem{MRZ}
Robert Marsh, Markus Reineke, and Andrei Zelevinsky.
\newblock Generalized associahedra via quiver representations.
\newblock {\em Trans. Amer. Math. Soc.}, 355(10):4171--4186, 2003.

\bibitem{Nak1}
Hiraku Nakajima.
\newblock Instantons on {ALE} spaces, quiver varieties, and {K}ac-{M}oody
  algebras.
\newblock {\em Duke Math. J.}, 76(2):365--416, 1994.

\bibitem{Nak2}
Hiraku Nakajima.
\newblock Quiver varieties and {K}ac-{M}oody algebras.
\newblock {\em Duke Math. J.}, 91(3):515--560, 1998.

\bibitem{Nak}
Hiraku Nakajima.
\newblock Quiver varieties and finite-dimensional representations of quantum
  affine algebras.
\newblock {\em J. Amer. Math. Soc.}, 14(1):145--238, 2001.

\bibitem{NakT}
Hiraku Nakajima.
\newblock Quiver varieties and tensor products.
\newblock {\em Invent. Math.}, 146(2):399--449, 2001.

\bibitem{NakExt}
Hiraku Nakajima.
\newblock Extremal weight modules of quantum affine algebras.
\newblock In {\em Representation theory of algebraic groups and quantum
  groups}, volume~40 of {\em Adv. Stud. Pure Math.}, pages 343--369. Math. Soc.
  Japan, Tokyo, 2004.

\bibitem{NakCA}
Hiraku Nakajima.
\newblock Quiver varieties and cluster algebras.
\newblock {\em Kyoto J. Math.}, 51(1):71--126, 2011.

\bibitem{OS}
Se-jin Oh and Travis Scrimshaw.
\newblock Denominators of {R}-matrices, higher {D}orey's rules and a
  generalization of {T}-systems for quantum affine algebras.
\newblock Preprint, \arxiv{2510.10874}v2, 2026.

\bibitem{Rei}
Markus Reineke.
\newblock Framed quiver moduli, cohomology, and quantum groups.
\newblock {\em J. Algebra}, 320(1):94--115, 2008.

\bibitem{VV}
M.~Varagnolo and E.~Vasserot.
\newblock Perverse sheaves and quantum {G}rothendieck rings.
\newblock In {\em Studies in memory of {I}ssai {S}chur ({C}hevaleret/{R}ehovot,
  2000)}, volume 210 of {\em Progr. Math.}, pages 345--365. Birkh\"{a}user
  Boston, Boston, MA, 2003.

\end{thebibliography}

\end{document}